\documentclass[final,12pt]{elsarticle}

\usepackage[utf8]{inputenc}
\usepackage[T1]{fontenc}
\usepackage[english]{babel}
\usepackage{amssymb,mathtools}
\usepackage{booktabs}
\usepackage{array}
\usepackage{tikz}
\usepackage{enumitem}
\usepackage{xcolor}
\usepackage{subcaption}
\usepackage{amsthm}
\usepackage{placeins}
\usepackage{url}

\newtheorem{theorem}{Theorem}[section]
\newtheorem{proposition}[theorem]{Proposition}
\newtheorem{lemma}[theorem]{Lemma}
\newtheorem{corollary}[theorem]{Corollary}

\newtheorem{example}[theorem]{Example}
\theoremstyle{remark}
\newtheorem{remark}[theorem]{Remark}

\begin{document}

\begin{frontmatter}

\title{Below-threshold Bistability and Implementation Lag
in a Simplex Model of Radical Vote-Share Dynamics}

\author{Alexander Omelchenko}
\ead{aomelchenko@constructor.university}
\address{Constructor University Bremen gGmbH,
Campus Ring 1, 28759 Bremen, Germany}

\begin{abstract}
We study a nonautonomous compartmental model on the probability simplex for
radical vote-share dynamics under delayed policy implementation. The model
distinguishes between a target regime and an effective regime that relaxes
toward it with finite speed, and it includes a persistent alienation reservoir
that can later be mobilised by radical actors. Two nonlinear effects drive the
analysis. First, in a symmetric benchmark, persistent alienation destroys the
global one-threshold picture: below the local instability threshold, the
radical-free centrist--alienated equilibrium may coexist with a stable positive
equilibrium, separated by a saddle branch born at a saddle-node bifurcation.
By hyperbolic persistence, the coexistence survives on an open neighbourhood
of the symmetric baseline; numerical exploration identifies a wider bistable
region under asymmetric perturbations.
Second, delayed implementation splits announced and effective threshold passage,
creating an explicit parameter-space lag between target and effective crossing
and a further state-space lag before a visible state response appears in
radical support. The analytical framework is based on a Perron--Frobenius threshold for
the frozen radical-free equilibrium, a geometric characterisation of positive
equilibria, local adiabatic tracking under slow uniformly subcritical drift,
and explicit delay bounds for transversal threshold passage. Under initial
matching, the parameter-space delay scales as \(O(\kappa_\theta^{-1})\); in a
monotone ramp example the bound is nearly attained, whereas the additional
state-space lag is substantially larger and depends much more weakly on the
implementation speed. The results imply that threshold restoration is
preventive rather than curative: once the state has entered the basin of the
radicalised attractor, returning below the local threshold need not restore the
radical-free regime.
\end{abstract}


\end{frontmatter}

\section{Introduction}\label{sec:intro}

Electoral support for radical parties often responds to institutional change
only after a delay, because reforms that are announced or adopted do not
instantly become operative in the environments that shape voter flows. This
paper studies that lag in a nonlinear compartmental model of radical
vote-share dynamics on the probability simplex. Radicalisation is
operationalised as the reallocation of aggregate voter support toward radical
blocs, not as individual attitude change. The state variables represent
left-radical support, right-radical support, centrist mainstream support, and
persistent latent alienation. Two mechanisms are central. First, chronic
alienation creates a reservoir that can later be mobilised into radical
support. Second, the political environment is split into a target regime
\(\theta_\star(t)\) and an effective regime \(\theta(t)\) that relaxes toward
it with finite implementation speed. The question is therefore not only
whether a frozen instability threshold exists, but when threshold restoration
becomes operative and whether it is sufficient to recover the radical-free
regime.

The target/effective distinction is meant in an operational rather than purely
mathematical sense. For example, a government may adopt a programme intended to
reduce exclusion in peripheral districts by funding counselling, language
training, housing support, or re-employment services. The target regime changes
on the date of adoption, but the effective regime changes only when budgets are
released, local offices hire staff, eligibility rules are enforced
consistently, and voters observe credible changes in their everyday
environment. During this interim period the announced policy environment and
the environment that actually affects alienation, re-engagement, and radical
recruitment can differ. Similar lags can arise after reforms of party
financing, media access, or online-platform regulation: legal rules may change
at once, whereas enforcement capacity, resource allocation, and public
perception adjust over months.

To formalise this distinction quantitatively, the frozen system provides
the local stability picture, but it does not by itself determine
transition timing under delayed implementation.
Let \(t_\star\) and \(t_{\mathrm{eff}}\)
denote the times at which the \emph{target} and \emph{effective} regimes,
respectively, cross the frozen instability threshold, so that
\(t_\star < t_{\mathrm{eff}}\). When a visible post-crossing state response
exists, let \(t_{\mathrm{obs}}\) denote its first occurrence; then
\begin{equation*}
  t_\star \;<\; t_{\mathrm{eff}} \;\le\; t_{\mathrm{obs}}.
\end{equation*}
The gap \(t_{\mathrm{eff}}-t_\star\) is a parameter-space lag caused
solely by finite implementation speed. The gap
\(t_{\mathrm{obs}}-t_{\mathrm{eff}}\) is an additional state-space lag
whenever \(t_{\mathrm{obs}}\) is finite; inside the bistable regime, a
trajectory that has entered the basin of the radicalised attractor need
not produce a finite \(t_{\mathrm{obs}}\) upon threshold restoration. Two questions therefore arise: (i) before \(t_\star\), when
does a trajectory continue to shadow a moving radical-free branch under
slow effective drift? (ii) how large is \(t_{\mathrm{eff}}-t_\star\),
and how does it scale with \(\kappa_\theta\)? The paper proves an
explicit \(O(\kappa_\theta^{-1})\) bound for the first gap and
documents numerically that the second gap \(t_{\mathrm{obs}}-t_{\mathrm{eff}}\)
is substantially larger and depends more weakly on \(\kappa_\theta\).
These questions place the model within the broader context of slowly
varying dynamical systems, stability-loss delay, and critical transitions
\cite{fenichel1979,kuehn2015,neishtadt2009,kuehn2011,cantisan2023}. At
the frozen level, the threshold analysis also has structural parallels
with Perron--Frobenius reproduction-number methods and backward
bifurcation phenomena in compartmental epidemiology
\cite{diekmann1990,vandendriessche2002,hadeler1995,castillo2004}.

The model builds on compartmental voter-share dynamics
\cite{volkening2020} and related mean-field approaches to opinion
formation and polarisation \cite{castellano2009,baumann2020}. Its empirical
motivation is anchored in three adjacent literatures. First, studies of mass
and affective polarisation document that electoral movement can be driven by
social identity, partisan sorting, and hostile out-party affect, not only by
changes in policy preferences
\cite{hetherington2009,iyengar2012,iyengarwestwood2015}. Second, work on
political support and democratic disaffection documents the erosive role of
distrust, dissatisfaction, and disengagement in advanced democracies
\cite{dalton2004,norris2011,mair2013}. Third, policy-implementation research
stresses that adoption does not equal operative implementation: effective
outcomes depend on administrative capacity, local enforcement, resource
allocation, and multi-actor implementation chains
\cite{pressman1984,sabatier1980,barrett2004}. Together these literatures
motivate the separation between the announced or intended institutional
levers collected in \(\theta_\star\) and their realised effects on
voter-flow rates collected in \(\theta\).

A concrete instance of the target/effective split arises in the German case
following the \textit{Integrationsgesetz} of August 2016. The law was adopted
with the explicit aim that earlier integration would be more effective, setting
a clear target regime for the political environment governing voter flows.
However, the Federal Audit Office (\textit{Bundesrechnungshof}) documented a
persistent implementation lag: average waiting times for integration courses
rose from 17 weeks in 2016 to 30.5 weeks in 2019, well above the six-week
target mandated by the law \cite{bundesrechnungshof2021}. Over the same
period, the AfD sustained substantial electoral support, entering the
Bundestag with 12.6\% in September 2017 \cite{bundeswahlleiter2017} and
remaining a significant force thereafter. We do not claim a direct causal
link between implementation lag and radical vote shares---the model isolates
a mechanism, not a specific causal pathway. Rather, this episode illustrates
the central distinction of the paper: a target regime (\(\theta_\star\)) may
be set at an identifiable moment while the effective regime (\(\theta\))
governing actual voter-flow dynamics lags behind by a measurable and
policy-relevant delay.

Persistent alienation changes the natural radical-free reference state.
Instead of the purely centrist equilibrium, the frozen system has a
centrist--alienated equilibrium with positive alienation level. Its
local stability is governed by a Perron--Frobenius threshold. That
threshold is the correct local object, but it need not classify the
entire frozen phase portrait. In the symmetric benchmark studied below,
a locally stable radical-free equilibrium can coexist with a stable
positive equilibrium below the local instability threshold. This loss of
global one-threshold behaviour is precisely why the adiabatic theorem
proved here is local in state space.

The paper has two main substantive contributions and two supporting analytical
results.
\begin{enumerate}[label=\textup{(\roman*)},leftmargin=2.5em]

\item \emph{Below-threshold bistability and hysteresis-type path dependence.}
In the symmetric benchmark, persistent alienation can generate a stable
positive equilibrium coexisting with the locally stable radical-free branch,
with the two attractors separated by a saddle born at a saddle-node
bifurcation of positive equilibria. Thus \(\mathcal R_{\mathrm{LA}}\) is a
local, not global, criterion: restoring the local threshold does not by
itself eliminate the radicalised attractor.
Analytically, bistability persists on an open neighbourhood of the
symmetric baseline by hyperbolic persistence; numerically, it survives
moderate departures from left--right symmetry
(Figure~\ref{fig:asymm_bistability_map}).

\item \emph{Two-stage delay under implementation lag.}
If the target regime crosses the frozen instability threshold transversally,
then the effective regime can cross only after an explicit delay of order
\(O(\kappa_\theta^{-1})\). This parameter-space delay
\(t_{\mathrm{eff}}-t_\star\) is distinct from the further state-space delay
\(t_{\mathrm{obs}}-t_{\mathrm{eff}}\), which is observed numerically to be
substantially larger and much less sensitive to \(\kappa_\theta\).
A robustness check across alternative visibility thresholds confirms
that this qualitative decomposition is independent of the specific
criterion used to define the visible state response (Table~\ref{tab:obs_delay}).

\item \emph{Perron--Frobenius threshold and frozen equilibrium geometry.}
For each frozen environment, local stability of the radical-free
centrist--alienated equilibrium is governed by the threshold
\(\mathcal R_{\mathrm{LA}}\). Positive frozen equilibria are organised by a
Perron--Frobenius surface together with a scalar compatibility relation,
separating the determination of \((C^*,A^*)\) from the reconstruction of
radical shares via the Perron eigenvector.

\item \emph{Local adiabatic tracking below threshold.}
Away from the threshold surface, and under uniformly subcritical slow
effective drift, trajectories that start near the radical-free branch remain
close to it, with a deviation bound proportional to the drift speed. This
statement is necessarily local in state space because the bistability
in~(i) rules out any global tracking theorem for all initial data.

\end{enumerate}
Items~(i)--(ii) are the substantive contributions; items~(iii)--(iv) provide
the analytical framework. The paper does not claim a new general-purpose
technique. The limiting case \(\eta=0\) recovers the threshold picture of
\cite{omelchenko2026threshold}; the alienation reservoir and the
target/effective split introduce two qualitative effects absent from that
model: below-threshold coexistence and a two-stage delay decomposition.
Although the application language is political, the mathematical structure is
broader: a conserved nonlinear compartmental system with a local frozen
threshold, slow parameter drift, and delayed effective response.

The remainder of the paper is organised as follows.
Section~\ref{sec:model} formulates the model and derives the frozen
local threshold. Section~\ref{sec:symm} studies the symmetric benchmark
and exhibits below-threshold bistability.
Section~\ref{sec:adiabatic} proves local adiabatic tracking away from
the threshold surface. Section~\ref{sec:delay} establishes explicit
bounds for delayed threshold passage of the effective regime.
Section~\ref{sec:discussion} concludes.

\section{Model and Frozen Threshold Geometry}\label{sec:model}

We consider a conserved electorate with four compartments: left-radical 
support $L$, right-radical support $R$, centrist mainstream $C$, and 
latent alienation $A$---a persistent reservoir of politically disaffected 
voters distinct from short-run abstention. The conservation law
\[
L(t)\ge0,\qquad R(t)\ge0,\qquad A(t)\ge0,\qquad
C(t)=1-L(t)-R(t)-A(t)\ge0
\]
defines the state space as the standard three-dimensional simplex
\begin{equation}\label{eq:state_space}
  \mathcal T
  :=
  \{(L,R,A)\in\mathbb R^3:\ L\ge0,\ R\ge0,\ A\ge0,\ L+R+A\le1\}.
\end{equation}

\begin{figure}[t]
\centering
\begin{tikzpicture}[
  box/.style={draw, rounded corners, minimum width=2.2cm,
              minimum height=0.9cm, align=center},
  arr/.style={->, >=stealth, thick},
  every node/.style={font=\small}
]
\node[box] (C) at (0,0)    {Centrist\\$C$};
\node[box] (L) at (-4,2.5) {Left-radical\\$L$};
\node[box] (R) at ( 4,2.5) {Right-radical\\$R$};
\node[box] (A) at (0,-2.5) {Alienated\\$A$};
\draw[arr] (C) -- node[left,  pos=0.15] {$\alpha_L CL + \gamma_{RL}CR$} (L);
\draw[arr] (C) -- node[right, pos=0.15] {$\alpha_R CR + \gamma_{LR}CL$} (R);
\draw[arr] (L) to[bend left=18]  node[above, pos=0.4] {$\mu_L$} (C);
\draw[arr] (R) to[bend right=18] node[above, pos=0.4] {$\mu_R$} (C);
\draw[arr] (C) -- node[right] {$\eta C$} (A);
\draw[arr] (A) to[bend left=15] node[left] {$\rho A$} (C);
\draw[arr] (A) to[bend left=50]
  node[left, pos=0.5] {$\delta_L AL$} (L);
\draw[arr] (A) to[bend right=50]
  node[right, pos=0.5] {$\delta_R AR$} (R);
\end{tikzpicture}
\caption{Compartment flow diagram with mean-field closures.
Flows \(\eta C\) and \(\rho A\) are absent from the \(\eta=0\) companion model.}
\label{fig:compartments}
\end{figure}
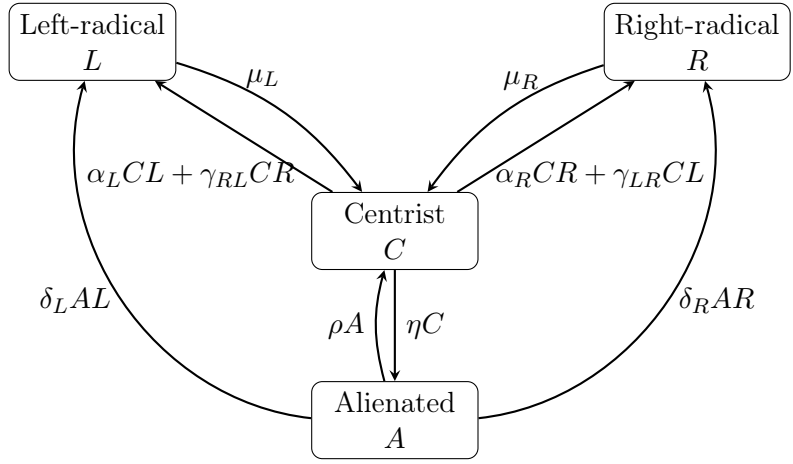

The voter-flow dynamics are driven by an \emph{effective} parameter path
\[
\theta(t)=
(\alpha_L,\alpha_R,\gamma_{LR},\gamma_{RL},
  \mu_L,\mu_R,\delta_L,\delta_R,\eta,\rho)(t)
\in\mathcal P,
\]
where $\mathcal P\subset\mathbb R^{10}$ is a compact convex set of admissible
regimes. For each $\theta\in\mathcal P$ we define
\begin{equation}\label{eq:matrices_model}
  K(\theta)=
  \begin{pmatrix}
    \alpha_L & \gamma_{RL}\\
    \gamma_{LR} & \alpha_R
  \end{pmatrix},
  \qquad
  M(\theta)=
  \begin{pmatrix}
    \mu_L & 0\\
    0 & \mu_R
  \end{pmatrix},
  \qquad
  D(\theta)=
  \begin{pmatrix}
    \delta_L & 0\\
    0 & \delta_R
  \end{pmatrix},
\end{equation}
where \(K(\theta)\) has strictly positive entries, \(M(\theta)\) and
\(D(\theta)\) have strictly positive diagonal entries, and
\(\eta(\theta),\rho(\theta)>0\). The matrix $K$ collects recruitment
from the centrist pool: the diagonal entries $\alpha_L$, $\alpha_R$
represent direct recruitment, and the off-diagonal entry $K_{ij}$
represents the rate at which growth in bloc $j$ drives centrists
toward bloc $i$ through reactive cross-polarisation; $M$ contains deradicalisation rates; $D$ contains
mobilisation rates from alienation into the two radical blocs; $\eta$ is
the chronic alienation rate; and $\rho$
governs the relaxation of alienation back toward effective political
engagement.

Let
\[
v(t):=(L(t),R(t))^\top,\qquad \mathbf 1:=(1,1)^\top,\qquad
C(t)=1-\mathbf 1^\top v(t)-A(t).
\]
The model reads
\begin{equation}\label{eq:model}
\begin{aligned}
  \dot v
  &=
  \bigl(C(t)K(\theta(t))+A(t)D(\theta(t))-M(\theta(t))\bigr)\,v,
  \\
  \dot A
  &=
  \eta(\theta(t))\,C(t)
  -
  A(t)\,\mathbf 1^\top D(\theta(t))\,v
  -
  \rho(\theta(t))\,A(t).
\end{aligned}
\end{equation}
The bilinear terms in \eqref{eq:model} are mean-field mass-action closures.
The term \(C K(\theta)v\) represents recruitment from the centrist pool through
interactions with existing radical milieus, whereas \(A D(\theta)v\) represents
mobilisation from latent alienation in the presence of radical organisation.
The term \(\eta(\theta)C\) is the new mechanism relative to the shock-based
companion model: even without an acute crisis, the centrist pool feeds a
persistent alienated reservoir that can later be mobilised into radical
support. This interpretation is consistent with compartmental election models
\cite{volkening2020} and with related mean-field approaches in social dynamics
\cite{castellano2009,baumann2020}. The political sociology literature
documents the empirical counterparts of these flows: chronic alienation
corresponds to the long-run growth in electoral non-participation
\cite{mair2013,norris2019}, while mobilisation from alienation into radical
support has been linked to the conversion of habitual non-voters into
populist-party supporters \cite{kriesi2012,norris2019}.

\begin{remark}[Political interpretation and empirical scope]
\label{rem:proxies}
Table~\ref{tab:proxies} summarises the political interpretation of each
parameter and a plausible observable proxy. The point of the mapping is not
calibration in the present paper but case anchoring. The quantities collected
in \(\theta\) should be read as policy-sensitive levers rather than fixed
traits of the electorate: \(\alpha_i\) and \(\gamma_{ij}\) summarise the
current permeability of the political and information environment to direct
recruitment and reactive polarisation, \(\mu_i\) summarises deradicalisation
or mainstream reabsorption capacity, \(\eta\) and \(\rho\) govern the
medium-run production and reabsorption of chronic alienation, and
\(\kappa_\theta\) measures the speed with which announced institutional change
becomes operative. The intended empirical scope is therefore a class of cases
in which reforms are decided or announced at a visible date but their
effective impact on voter flows is delayed by implementation, enforcement, or
behavioural adjustment. Within that scope, \(1/\kappa_\theta\) sets the scale
of the implementation-lag bound derived below.
\end{remark}

\begin{table}[t]
\centering
\caption{Parameters, political interpretation, and plausible empirical
proxies.}
\label{tab:proxies}
\begin{tabular}{@{}llp{5.5cm}@{}}
\toprule
Parameter & Political interpretation & Plausible proxy \\
\midrule
\(\alpha_{L},\alpha_{R}\) & Direct centrist-to-radical recruitment &
  Election-to-election change in radical party vote share
  \cite{kriesi2012} \\
\(\gamma_{LR},\gamma_{RL}\) & Reactive cross-polarisation &
  Correlated shifts between left- and right-radical vote shares
  across electoral cycles \\
\(\mu_{L},\mu_{R}\) & Party exit / deradicalisation rate &
  Decline in radical party membership or vote share between elections \\
\(\eta\) & Chronic alienation inflow &
  Annual rise in habitual non-voters among previously active cohorts
  \cite{mair2013} \\
\(\rho\) & Re-engagement rate &
  Return rate to centrist-party voting in panel surveys
  \cite{norris2019} \\
\(\delta_{L},\delta_{R}\) & Mobilisation from alienation &
  Conversion of habitual non-voters into radical-party voters
  \cite{norris2019,kriesi2012} \\
\(\kappa_\theta\) & Institutional implementation speed &
  Inverse of mean lag between policy adoption and measurable
  change in party-system regulation or media access \\
\bottomrule
\end{tabular}
\end{table}

Throughout the paper, the state equations are driven by the effective regime
$\theta(\cdot)$. In Sections~\ref{sec:adiabatic}--\ref{sec:delay} this regime
will itself be coupled to a \emph{target} regime $\theta_\star(t)$ through the
relaxation law
\begin{equation}\label{eq:theta_relax}
  \dot\theta=\kappa_\theta\bigl(\theta_\star(t)-\theta(t)\bigr),
  \qquad \kappa_\theta>0.
\end{equation}
The distinction between $\theta_\star$ and $\theta$ is the mathematical
representation of implementation lag: political or institutional change may be
announced or intended before it becomes effective in the electorate dynamics.
For the moment, however, $\theta(\cdot)$ may be any absolutely continuous path
in $\mathcal P$.

\begin{proposition}[Forward invariance and global well posedness]
\label{prop:forward_invariant}
For every absolutely continuous parameter path $\theta:[0,\infty)\to\mathcal P$
and every initial condition $(L_0,R_0,A_0)\in\mathcal T$, the system
\eqref{eq:model} has a unique global solution
\[
  (L(t),R(t),A(t))\in\mathcal T
  \qquad\text{for all }t\ge0.
\]
\end{proposition}

\begin{proof}
See \ref{app:prop_forward_invariant}.
\end{proof}

We now pass to the \emph{frozen} problem, in which $\theta(t)\equiv\bar\theta$
is held constant. The frozen analysis identifies the equilibrium structure
and local stability thresholds that serve as the reference geometry for
the slow-drift and delay results of
Sections~\ref{sec:adiabatic}--\ref{sec:delay}. Fix $\bar\theta\in\mathcal P$
and replace $\theta(t)$ in \eqref{eq:model} by the constant value $\bar\theta$.
Writing, for brevity,
\[
K:=K(\bar\theta),\qquad M:=M(\bar\theta),\qquad D:=D(\bar\theta),\qquad
\eta:=\eta(\bar\theta),\qquad \rho:=\rho(\bar\theta),
\]
the frozen system becomes
\begin{equation}\label{eq:model_frozen}
\begin{aligned}
  \dot v &= (CK+AD-M)v,\\
  \dot A &= \eta C-A\,\mathbf 1^\top Dv-\rho A,
\end{aligned}
\qquad
C=1-\mathbf 1^\top v-A.
\end{equation}
Because the chronic alienation flow $\eta C$ remains active even in the absence
of radical support, the natural radical-free reference state is no longer the
purely centrist point. Instead, one obtains a centrist--alienated equilibrium.

\begin{theorem}[Frozen local threshold]\label{thm:frozen_threshold}
For each fixed $\bar\theta\in\mathcal P$, the frozen system
\eqref{eq:model_frozen} has a unique radical-free equilibrium
\begin{equation}\label{eq:E_dagger}
E^\dagger(\bar\theta)=(0,0,A^\dagger),
\qquad
A^\dagger=\frac{\eta}{\eta+\rho}\in(0,1),
\qquad
C^\dagger=1-A^\dagger=\frac{\rho}{\eta+\rho}\in(0,1).
\end{equation}
The Jacobian at \(E^\dagger(\bar\theta)\) is block triangular, with radical block
\begin{equation}\label{eq:Jrad}
  J_{\mathrm{rad}}(\bar\theta)
  =
  C^\dagger K(\bar\theta)+A^\dagger D(\bar\theta)-M(\bar\theta).
\end{equation}
Define
\begin{equation}\label{eq:RLA}
  \mathcal R_{\mathrm{LA}}(\bar\theta)
  :=
  \lambda_{\mathrm{PF}}\!\left(
    M(\bar\theta)^{-1}
    \bigl[
      C^\dagger K(\bar\theta)+A^\dagger D(\bar\theta)
    \bigr]
  \right).
\end{equation}
Then:
\[
\mathcal R_{\mathrm{LA}}(\bar\theta)<1
\iff
E^\dagger(\bar\theta)\ \text{is locally asymptotically stable},
\]
\[
\mathcal R_{\mathrm{LA}}(\bar\theta)=1
\iff
E^\dagger(\bar\theta)\ \text{is nonhyperbolic},
\]
\[
\mathcal R_{\mathrm{LA}}(\bar\theta)>1
\iff
E^\dagger(\bar\theta)\ \text{is unstable}.
\]
\end{theorem}

\begin{proof}
On the radical-free set $L=R=0$, one has $C=1-A$, and the $A$-equation reduces
to
\[
  \dot A=\eta(1-A)-\rho A,
\]
which has the unique equilibrium \eqref{eq:E_dagger}. Linearising
\eqref{eq:model_frozen} at $E^\dagger(\bar\theta)$ yields a block-triangular
Jacobian of the form
\[
  J(E^\dagger)=
  \begin{pmatrix}
    J_{\mathrm{rad}}(\bar\theta) & 0\\
    * & -(\eta+\rho)
  \end{pmatrix},
\]
so local stability is determined entirely by the $2\times2$ radical block
\eqref{eq:Jrad}. The matrix
\[
  N^\dagger:=C^\dagger K(\bar\theta)+A^\dagger D(\bar\theta)
\]
is strictly positive, hence
\[
  J_{\mathrm{rad}}(\bar\theta)=N^\dagger-M(\bar\theta)
\]
is an irreducible Metzler matrix. By the standard
Perron--Frobenius/$M$-matrix criterion,
\[
s\!\left(J_{\mathrm{rad}}(\bar\theta)\right)<0,\ =0,\ >0
\quad\Longleftrightarrow\quad
\lambda_{\mathrm{PF}}\!\left(M(\bar\theta)^{-1}N^\dagger\right)<1,\ =1,\ >1,
\]
which is precisely the stated trichotomy.
\end{proof}

\begin{remark}[Limit \(\eta\downarrow0\) and connection to the companion model]
Although the admissible regime set \(\mathcal P\) is defined by the standing
assumption \(\eta>0\), it is useful to record the limiting relation with the
companion model. If all frozen parameters except \(\eta\) are kept fixed and
\(\eta\downarrow0\), then from \eqref{eq:E_dagger} one obtains
\[
  A^\dagger\to0,
  \qquad
  C^\dagger\to1,
\]
so the radical-free centrist--alienated equilibrium tends to the purely
centrist equilibrium. At the same time, the threshold formula \eqref{eq:RLA}
reduces to
\[
  \lambda_{\mathrm{PF}}\!\bigl(M^{-1}K\bigr),
\]
which is exactly the recruitment-versus-deradicalisation threshold of the
companion model \cite{omelchenko2026threshold}. In the symmetric
specialisation introduced in Section~\ref{sec:symm}, the corresponding
threshold further reduces to
\[
  \frac{\beta}{\mu}.
\]
Thus the present framework extends the \(\eta=0\) model by allowing a
persistent alienated reservoir.
\end{remark}

Theorem~\ref{thm:frozen_threshold} gives the correct \emph{local} threshold at
the radical-free centrist--alienated branch. Although the present setting is
political rather than epidemiological, the quantity
\(\mathcal R_{\mathrm{LA}}\) plays the same formal role as the basic
reproduction number in compartmental disease models
\cite{diekmann1990,vandendriessche2002}: local stability of the radical-free
branch is governed by whether the Perron root of the effective recruitment
operator exceeds one. This local threshold, however, need not classify the
entire frozen phase portrait. To see why, it is useful to describe the
geometry of positive frozen equilibria. The next proposition shows that such
equilibria are not arbitrary: their centrist and alienation shares must satisfy
one algebraic surface condition, and their radical shares are then determined
by the corresponding Perron eigenvector.

\begin{proposition}[Perron--Frobenius geometry of positive frozen equilibria]
\label{prop:pf_surface}
Fix $\bar\theta\in\mathcal P$, and retain the shorthand
\[
K:=K(\bar\theta),\qquad M:=M(\bar\theta),\qquad D:=D(\bar\theta),\qquad
\eta:=\eta(\bar\theta),\qquad \rho:=\rho(\bar\theta).
\]
For $(C,A)$ with
\[
  C>0,\qquad A\ge0,\qquad C+A<1,
\]
define
\[
  q(C,A):=1-C-A,
  \qquad
  B(C,A):=M^{-1}(CK+AD).
\]
Whenever $\lambda_{\mathrm{PF}}(B(C,A))=1$, let $u(C,A)\gg0$ denote the
corresponding Perron eigenvector normalised by
\[
  \mathbf 1^\top u(C,A)=1.
\]
Then a point $(L^*,R^*,A^*)\in\mathcal T$ with
\[
  L^*>0,\qquad R^*>0,\qquad C^*:=1-L^*-R^*-A^*>0
\]
is an equilibrium of \eqref{eq:model_frozen} if and only if
\begin{equation}\label{eq:pf_surface}
  \lambda_{\mathrm{PF}}\bigl(B(C^*,A^*)\bigr)=1
\end{equation}
and
\begin{equation}\label{eq:pf_compatibility}
  \eta C^*
  =
  A^*\Bigl[\rho+q(C^*,A^*)\,\mathbf 1^\top D\,u(C^*,A^*)\Bigr],
\end{equation}
with
\begin{equation}\label{eq:pf_reconstruction}
  \binom{L^*}{R^*}=q(C^*,A^*)\,u(C^*,A^*).
\end{equation}
In particular, every positive frozen equilibrium lies on the
Perron--Frobenius surface \eqref{eq:pf_surface}, and the remaining equilibrium
condition reduces to the scalar compatibility relation
\eqref{eq:pf_compatibility}. Equations \eqref{eq:pf_surface} and
\eqref{eq:pf_compatibility} form a closed system for the two unknowns
$(C^*,A^*)$: one seeks a pair for which the Perron root of $B(C^*,A^*)$
equals one and the compatibility condition holds simultaneously; once
such a pair is found, the radical shares are recovered from
\eqref{eq:pf_reconstruction}.
\end{proposition}

\begin{proof}
Let $(L^*,R^*,A^*)$ be a positive frozen equilibrium and set
\[
  v^*:=\binom{L^*}{R^*},
  \qquad
  C^*:=1-L^*-R^*-A^*.
\]
From the $v$-equation in \eqref{eq:model_frozen} we obtain
\[
  (C^*K+A^*D-M)v^*=0,
\]
hence
\[
  B(C^*,A^*)\,v^*=v^*.
\]
Since $v^*\gg0$, the Perron--Frobenius theorem implies
$\lambda_{\mathrm{PF}}(B(C^*,A^*))=1$. Normalising
\[
  u(C^*,A^*):=\frac{v^*}{L^*+R^*}
\]
gives the positive Perron eigenvector with
$\mathbf 1^\top u(C^*,A^*)=1$, and since
\[
  L^*+R^*=1-C^*-A^*=q(C^*,A^*),
\]
we obtain \eqref{eq:pf_reconstruction}. Substituting
\[
  v^*=q(C^*,A^*)\,u(C^*,A^*)
\]
into the frozen $A$-equation yields \eqref{eq:pf_compatibility}.

Conversely, suppose that $(C^*,A^*)$ satisfies \eqref{eq:pf_surface} and
\eqref{eq:pf_compatibility}, and define $(L^*,R^*)$ by
\eqref{eq:pf_reconstruction}. Since $u(C^*,A^*)$ is a Perron eigenvector of
$B(C^*,A^*)$ with eigenvalue $1$, we have
\[
  (C^*K+A^*D-M)\binom{L^*}{R^*}=0.
\]
The frozen $A$-equation then holds by \eqref{eq:pf_compatibility}, so
$(L^*,R^*,A^*)$ is an equilibrium of \eqref{eq:model_frozen}.
\end{proof}

In the symmetric benchmark studied next, the compatibility condition
\eqref{eq:pf_compatibility} reduces to an explicit quadratic, making it
possible to show concretely that the local threshold need not be globally
decisive.

\section{Symmetric Benchmark: Local Threshold, Positive Equilibria, and
Bistability}\label{sec:symm}

To understand concretely how the frozen local threshold interacts with the
global equilibrium geometry, we now impose a symmetric benchmark. This should
not be read as an exact empirical claim about European party systems, but as a
reference case in which the two radical blocs have the same recruitment,
deradicalisation, and alienation-mobilisation structure. In this setting the
general Perron--Frobenius surface description from
Proposition~\ref{prop:pf_surface} collapses to an explicit scalar equation, and
the loss of global one-threshold behaviour becomes completely transparent.

We assume
\begin{equation}\label{eq:symm_assumptions}
  \alpha_L=\alpha_R=\alpha,
  \qquad
  \gamma_{LR}=\gamma_{RL}=\gamma,
  \qquad
  \mu_L=\mu_R=\mu,
  \qquad
  \delta_L=\delta_R=\delta,
\end{equation}
and introduce the total centrist-to-radical recruitment intensity
\begin{equation}\label{eq:beta_symm}
  \beta:=\alpha+\gamma.
\end{equation}
Throughout this section the environment is frozen and symmetric.

Under \eqref{eq:symm_assumptions}, the diagonal
\[
  \Delta:=\{(L,R,A)\in\mathcal T:\ L=R\}
\]
is forward invariant. Indeed, setting \(H:=L-R\) and subtracting the frozen
\(R\)-equation from the frozen \(L\)-equation gives
\[
  \dot H=\bigl((\alpha-\gamma)C+\delta A-\mu\bigr)H.
\]
Hence \(H(0)=0\) implies \(H(t)=0\) for all forward time. Writing
\[
  L=R=P,
  \qquad
  C=1-2P-A,
\]
the dynamics on \(\Delta\) reduces to the planar system
\begin{equation}\label{eq:model_symm}
\begin{aligned}
  \dot P
  &=P\bigl[(\beta-\mu)-2\beta P+(\delta-\beta)A\bigr],\\
  \dot A
  &=\eta(1-2P-A)-(2\delta P+\rho)A,
\end{aligned}
\end{equation}
on the triangle
\begin{equation}\label{eq:F_symm}
  \mathcal F
  :=
  \{(P,A)\in\mathbb R^2:\ P\ge0,\ A\ge0,\ 2P+A\le1\}.
\end{equation}
By Proposition~\ref{prop:forward_invariant}, \(\mathcal F\) is forward
invariant.

The radical-free frozen equilibrium from
Theorem~\ref{thm:frozen_threshold} becomes
\begin{equation}\label{eq:E0_symm}
  E_0^{\mathrm{sym}}=(0,A^\dagger),
  \qquad
  A^\dagger=\frac{\eta}{\eta+\rho},
  \qquad
  C^\dagger=\frac{\rho}{\eta+\rho}.
\end{equation}
Its local threshold now admits a closed form.

\begin{proposition}[Symmetric reduction and frozen equilibrium structure]
\label{prop:symm_benchmark}
For the reduced system \eqref{eq:model_symm} the following statements hold.
Parts~\textup{(i)--(iii)} characterise the equilibrium structure, while
part~\textup{(iv)} records the Jacobian at an arbitrary positive equilibrium;
this will be used to determine stability in
Example~\ref{ex:subcritical_coexistence}.

\begin{enumerate}[label=\textup{(\roman*)},leftmargin=2.5em]
  \item The equilibrium \(E_0^{\mathrm{sym}}\) is locally asymptotically
  stable, nonhyperbolic, or unstable according as
  \begin{equation}\label{eq:RLA_symm}
    \mathcal R_{\mathrm{LA}}^{\mathrm{sym}}
    :=
    \frac{\rho\beta+\eta\delta}{\mu(\eta+\rho)}
  \end{equation}
  is \(<1\), \(=1\), or \(>1\).

  \item A point \((P,A)\in\mathcal F\) with \(P>0\) is an equilibrium of
  \eqref{eq:model_symm} if and only if
  \begin{equation}\label{eq:AofP_symm}
    A=A(P):=\frac{\eta(1-2P)}{\eta+\rho+2\delta P}
  \end{equation}
  and \(P\in(0,\tfrac12)\) satisfies
  \begin{equation}\label{eq:Q_symm}
    Q(P)=0,
  \end{equation}
  where
  \begin{equation}\label{eq:Q_explicit_symm}
    Q(P)
    :=
    4\beta\delta P^2
    +2\bigl[\beta(\rho-\delta)+\delta(\eta+\mu)\bigr]P
    +\mu(\eta+\rho)-(\rho\beta+\eta\delta).
  \end{equation}

  \item The endpoint values satisfy
  \begin{equation}\label{eq:Q_endpoints_symm}
    Q(0)=\mu(\eta+\rho)\bigl(1-\mathcal R_{\mathrm{LA}}^{\mathrm{sym}}\bigr),
    \qquad
    Q\!\left(\frac12\right)=\mu(\delta+\eta+\rho)>0.
  \end{equation}

  \item At any positive equilibrium of \eqref{eq:model_symm}, the Jacobian
  simplifies to
  \begin{equation}\label{eq:J_positive_symm}
    J(P,A)=
    \begin{pmatrix}
      -2\beta P & P(\delta-\beta)\\[3pt]
      -2(\eta+\delta A) & -(\eta+\rho+2\delta P)
    \end{pmatrix}.
  \end{equation}
\end{enumerate}
\end{proposition}

\begin{proof}
See \ref{app:prop_symm_benchmark}.
\end{proof}

The next proposition shows that every positive equilibrium of the full frozen
symmetric model is necessarily symmetric, so the planar reduction captures all
positive frozen equilibria.

\begin{proposition}[All positive frozen equilibria are symmetric]
\label{prop:all_positive_symm}
Assume \eqref{eq:symm_assumptions}. Let
\[
  E^*=(L^*,R^*,A^*)\in\mathcal T
\]
be a positive equilibrium of the full frozen symmetric model, with
\[
  L^*>0,\qquad R^*>0,\qquad A^*\ge0,\qquad
  C^*:=1-L^*-R^*-A^*>0.
\]
Then
\[
  L^*=R^*.
\]
Equivalently, every positive equilibrium of the full frozen symmetric model
lies on the invariant diagonal
\[
  \Delta=\{(L,R,A)\in\mathcal T:\ L=R\}.
\]
\end{proposition}

\begin{proof}
See \ref{app:prop_all_positive_symm}.
\end{proof}

To recover the stability type in the full three-dimensional system, it remains
to control the antisymmetric mode. We therefore record the following lifting
result.

\begin{proposition}[Lifting planar stability to the full symmetric model]
\label{prop:lift_symm_stability}
Let \((P^*,A^*)\) be a positive equilibrium of the reduced system
\eqref{eq:model_symm}, and let
\[
  E^*=(L^*,R^*,A^*)=(P^*,P^*,A^*)
\]
be the corresponding equilibrium of the full frozen symmetric model. Then the
linearisation at \(E^*\) splits into the planar block \(J(P^*,A^*)\) from
\eqref{eq:J_positive_symm} and one antisymmetric mode \(H=L-R\), whose
eigenvalue is
\begin{equation}\label{eq:lambda_H_symm}
  \lambda_H
  =
  (\alpha-\gamma)C^*+\delta A^*-\mu
  =
  -2\gamma C^*
  <0,
\end{equation}
where \(C^*=1-2P^*-A^*\). Consequently, the full three-dimensional stability
type of \(E^*\) is obtained from the planar one by adding one stable
eigenvalue. In particular, a planar sink is also asymptotically stable in the
full symmetric model, while a planar saddle with one unstable direction remains
a saddle with one unstable direction in the full symmetric model.
\end{proposition}

\begin{proof}
See \ref{app:prop_lift_symm_stability}.
\end{proof}

\begin{corollary}[Subcriticality does not imply global attraction]
\label{cor:symm_local_not_global}
If \(\mathcal R_{\mathrm{LA}}^{\mathrm{sym}}<1\) and the quadratic
\eqref{eq:Q_symm} has a root in \((0,\tfrac12)\), then the radical-free
equilibrium \(E_0^{\mathrm{sym}}\) is locally asymptotically stable but not
globally attractive on \(\mathcal F\). In particular, the threshold
\(\mathcal R_{\mathrm{LA}}^{\mathrm{sym}}=1\) does not yield a global
one-regime classification for the frozen symmetric dynamics.
\end{corollary}

\begin{proof}
Immediate from Proposition~\ref{prop:symm_benchmark}\textup{(i)--(ii)}.
\end{proof}

Corollary~\ref{cor:symm_local_not_global} shows that the frozen threshold is a
\emph{local} criterion only: it governs small-amplitude behaviour near
\(E_0^{\mathrm{sym}}\), but does not preclude a coexisting positive attractor
at finite distance. The next example shows concretely that this situation
indeed occurs.

\begin{remark}[Backward bifurcation analogue and bifurcation type]
\label{rem:backward_bif}
The subcritical coexistence of
Corollary~\ref{cor:symm_local_not_global} is the political analogue of
\emph{backward bifurcation} in compartmental epidemic models
\cite{hadeler1995,castillo2004}: a locally stable disease-free equilibrium can
coexist with a stable endemic equilibrium even when the reproduction number is
below one. The functional similarity is the presence of a reservoir that is not
currently infective, or not currently radical, but can be returned to the active
class. In epidemic models, reinfection and waning immunity can keep a pool of
individuals available for renewed transmission. In the present model, the
alienated compartment \(A\) plays the analogous dynamical role. Alienated voters
are not counted as radical support, but the bilinear mobilisation term \(ADv\)
turns this reservoir into new radical recruitment whenever organised radical
support \(v\) is already present. Thus lowering direct recruitment below the
local threshold can suppress small radical seeds while still leaving a
finite-amplitude loop
\[
  C \longrightarrow A \longrightarrow \{L,R\}
\]
that maintains the positive attractor. The bifurcation creating the pair
$(E_-^{\mathrm{sym}}, E_+^{\mathrm{sym}})$ is a \emph{saddle-node bifurcation
of positive equilibria}, occurring as parameters cross the discriminant curve
\(\operatorname{disc}_\beta(Q)=0\) in parameter space, where \(Q\) is the
quadratic \eqref{eq:Q_explicit_symm}.
\end{remark}

\begin{example}[Subcritical coexistence and bistability]
\label{ex:subcritical_coexistence}
Take
\begin{equation}\label{eq:subcritical_example_params}
  (\beta,\mu,\delta,\eta,\rho)
  =
  \left(
    2,\ 1,\ \frac35,\ \frac15,\ \frac1{20}
  \right).
\end{equation}
Then
\[
  \mathcal R_{\mathrm{LA}}^{\mathrm{sym}}
  =
  \frac{\rho\beta+\eta\delta}{\mu(\eta+\rho)}
  =
  \frac{22}{25}
  <1,
\]
so the radical-free equilibrium is locally asymptotically stable. At the same
time,
\[
  Q(P)=\frac1{100}\bigl(480P^2-76P+3\bigr),
\]
whose two roots in \((0,\tfrac12)\) are
\[
  P_-=\frac{3}{40},
  \qquad
  P_+=\frac{1}{12}.
\]
The corresponding alienation levels from \eqref{eq:AofP_symm} are
\[
  A_-=\frac12,
  \qquad
  A_+=\frac{10}{21}.
\]
Thus the frozen symmetric system has two positive equilibria
\[
  E_-^{\mathrm{sym}}=\left(\frac{3}{40},\frac12\right),
  \qquad
  E_+^{\mathrm{sym}}=\left(\frac{1}{12},\frac{10}{21}\right),
\]
in the \((P,A)\)-plane.

Using \eqref{eq:J_positive_symm}, one finds
\[
  \det J(E_-^{\mathrm{sym}})=-\frac{3}{1000}<0,
\]
so \(E_-^{\mathrm{sym}}\) is a saddle, whereas
\[
  \operatorname{tr}J(E_+^{\mathrm{sym}})=-\frac{41}{60}<0,
  \qquad
  \det J(E_+^{\mathrm{sym}})=\frac{1}{300}>0,
\]
so \(E_+^{\mathrm{sym}}\) is locally asymptotically stable. By
Proposition~\ref{prop:all_positive_symm}, every positive equilibrium of the
full frozen symmetric model lies on the diagonal \(L=R\), and by
Proposition~\ref{prop:lift_symm_stability} the planar stability types transfer
directly to the full three-dimensional system. Therefore the frozen symmetric
benchmark exhibits genuine bistability in the full model: the radical-free
equilibrium and the upper positive equilibrium are both locally asymptotically
stable, and their basins are separated by the saddle branch.
Figure~\ref{fig:bistability} illustrates the corresponding basin structure in
the \((P,A)\)-plane on the invariant diagonal.
\end{example}

\begin{figure}[t]
\centering
\includegraphics[width=0.65\textwidth]{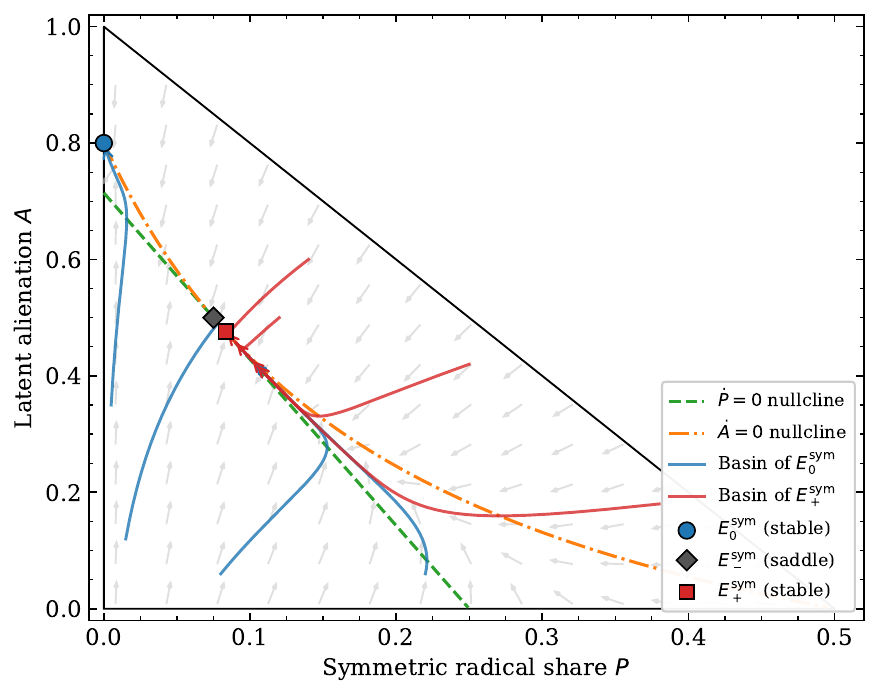}
\caption{Phase portrait of \eqref{eq:model_symm} for
Example~\ref{ex:subcritical_coexistence}
(\(\mathcal R_{\mathrm{LA}}^{\mathrm{sym}}=22/25<1\)).
Blue/red: trajectories from basins of \(E_0^{\mathrm{sym}}\) and
\(E_+^{\mathrm{sym}}\); grey diamond: saddle \(E_-^{\mathrm{sym}}\);
dashed/dash-dotted: nullclines.}
\label{fig:bistability}
\end{figure}

Example~\ref{ex:subcritical_coexistence} sharpens the message of
Corollary~\ref{cor:symm_local_not_global}. The issue is not merely that
positive equilibria may exist below the local threshold, but that a locally
stable radical-free regime may coexist with a distinct stable radicalised
regime. In political terms, a centrist--alienated background that is locally
resilient to small radical seeds need not exclude a self-sustaining radicalised
state elsewhere in the simplex.
Figure~\ref{fig:bistability_timeseries} complements the phase portrait by
showing time-domain trajectories from the two basins of attraction and from
initial data near the separatrix.

\begin{figure}[t]
\centering
\includegraphics[width=0.92\textwidth]{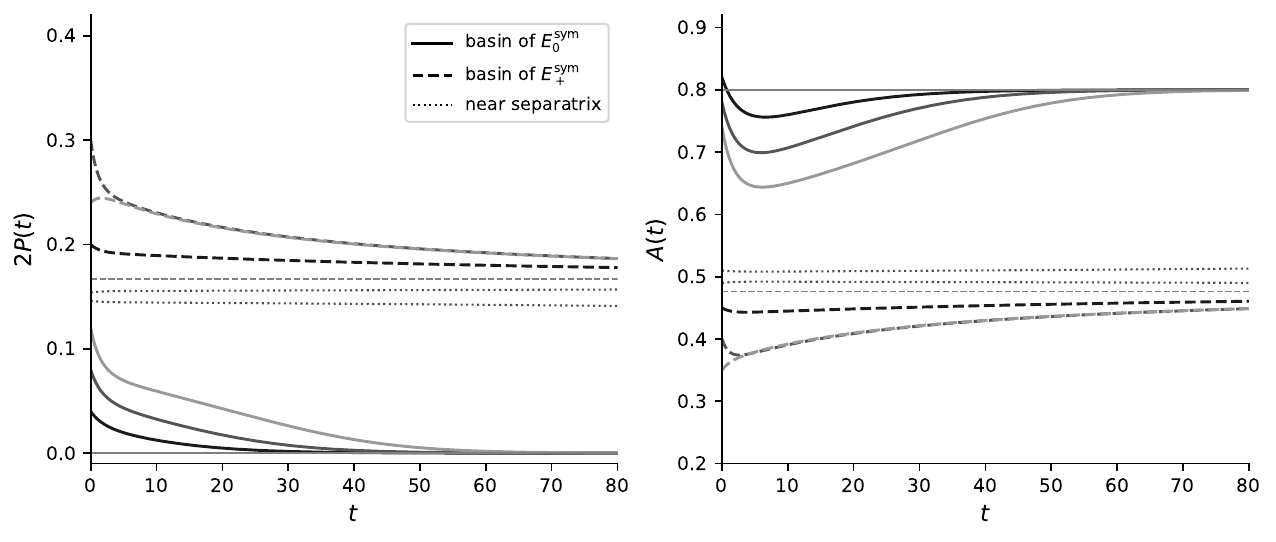}
\caption{Time-domain trajectories for Example~\ref{ex:subcritical_coexistence}.
Left: \(2P(t)\). Right: \(A(t)\). Solid/dashed/dotted: basins of
\(E_0^{\mathrm{sym}}\), \(E_+^{\mathrm{sym}}\), and near the separatrix.
Horizontal lines: equilibrium levels.}
\label{fig:bistability_timeseries}
\end{figure}

The bistability window of Example~\ref{ex:subcritical_coexistence} can be
located explicitly by treating \(\beta\) as a bifurcation parameter while
holding \((\mu,\delta,\eta,\rho)\) fixed at the values in
\eqref{eq:subcritical_example_params}. The local threshold of the
radical-free branch occurs at
\[
  \beta_c=\frac{\mu(\eta+\rho)-\eta\delta}{\rho}=\frac{13}{5}=2.6,
\]
whereas the discriminant of the quadratic \(Q(P)\), as a function of
\(\beta\), is
\[
  \operatorname{disc}_\beta(Q)
  =\frac{1}{2500}\bigl(4225\beta^2-11040\beta+5184\bigr).
\]
The admissible positive branches in \((0,1/2)\) are created at the saddle-node
value
\[
  \beta_{\mathrm{SN}}
  =
  \frac{1104+24\sqrt{595}}{845}
  \approx 1.999317.
\]
Hence the frozen symmetric system has a bistability window
\[
  \beta_{\mathrm{SN}}<\beta<\beta_c,
\]
in which the radical-free equilibrium and the upper positive equilibrium are
both locally stable, separated by a saddle branch.
Figure~\ref{fig:bifurcation_beta} displays the full branch structure.

\begin{figure}[t]
\centering
\includegraphics[width=0.72\textwidth]{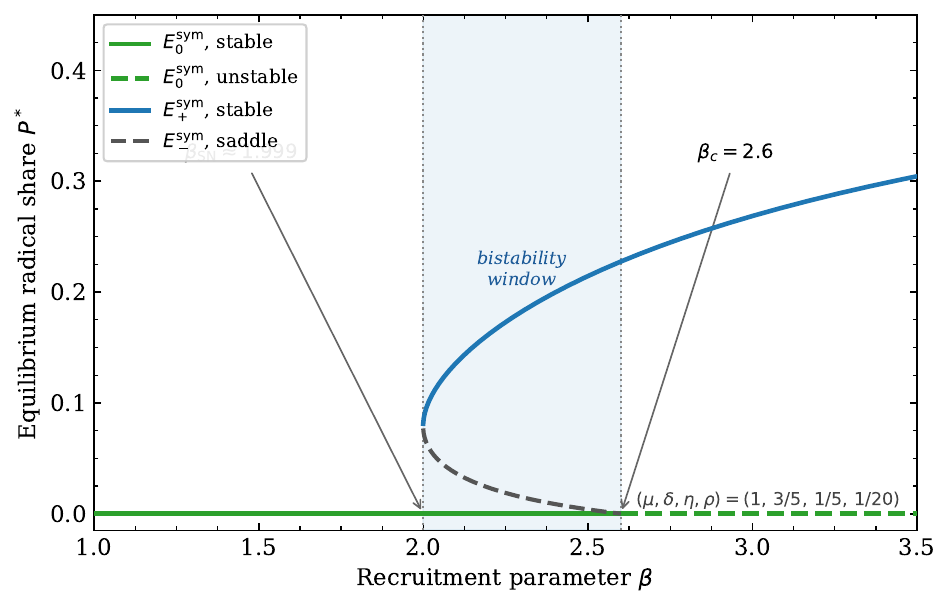}
\caption{Bifurcation diagram of \eqref{eq:model_symm} vs.\ \(\beta\),
fixed \((\mu,\delta,\eta,\rho)=(1,3/5,1/5,1/20)\). Solid: stable;
dashed: unstable/saddle. Bistability window:
\(\beta_{\mathrm{SN}}\approx 1.999317 < \beta < \beta_c=13/5\).}
\label{fig:bifurcation_beta}
\end{figure}

\begin{remark}[Persistence under weak asymmetry]
\label{rem:asymm_persistence}
At the parameter values of Example~\ref{ex:subcritical_coexistence}, the
radical-free equilibrium, the positive stable equilibrium, and the separating
saddle are all hyperbolic in the full symmetric frozen model by
Theorem~\ref{thm:frozen_threshold} and
Proposition~\ref{prop:lift_symm_stability}. Therefore, by the implicit function
theorem and continuity of eigenvalues, these equilibria persist with the same
stability types under sufficiently small perturbations of the symmetry
assumptions \eqref{eq:symm_assumptions}. In this sense, the coexistence pattern
is not an artifact of exact symmetry.
\end{remark}

\begin{figure}[t]
\centering
\includegraphics[width=0.62\textwidth]{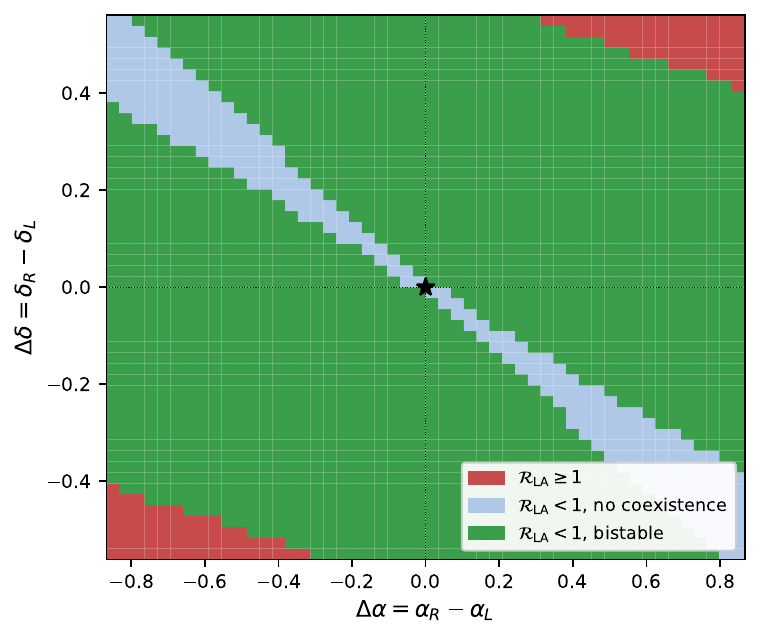}
\caption{Bistability map in the \((\Delta\alpha,\Delta\delta)\)-plane
relative to the symmetric benchmark
\((\alpha=1,\delta=0.6,\mu=1,\eta=0.2,\rho=0.05)\).
Green: \(\mathcal R_{\mathrm{LA}}<1\), bistable.
Blue: \(\mathcal R_{\mathrm{LA}}<1\), no coexistence.
Red: \(\mathcal R_{\mathrm{LA}}\ge1\). Star: symmetric baseline.}
\label{fig:asymm_bistability_map}
\end{figure}

Figure~\ref{fig:asymm_bistability_map} maps the bistability region
across a $(\Delta\alpha,\Delta\delta)$-neighbourhood of the symmetric
baseline. By Remark~\ref{rem:asymm_persistence}, bistability persists analytically
on an open neighbourhood of the symmetric baseline. Within the plotted
window, the numerically identified bistable region occupies most of the
subcritical area, confirming that below-threshold coexistence is not
confined to the exactly symmetric case. The
no-coexistence band visible along the anti-diagonal
($\Delta\delta \approx -0.6\,\Delta\alpha$) corresponds to parameter
combinations in which higher recruitment asymmetry in one bloc is offset
by lower mobilisation asymmetry, removing positive equilibria while
keeping $\mathcal R_{\mathrm{LA}}<1$.

The bistability map in Figure~\ref{fig:asymm_bistability_map} was
computed on a $50\times50$ grid over
$(\Delta\alpha,\Delta\delta)\in[-0.85,0.85]\times[-0.55,0.55]$, with
$(\alpha_L,\alpha_R)=(1+\Delta\alpha,1-\Delta\alpha)$ and
$(\delta_L,\delta_R)=(0.6+\Delta\delta,0.6-\Delta\delta)$. At each grid point,
the Perron--Frobenius threshold $\mathcal R_{\mathrm{LA}}$ was evaluated
analytically. If $\mathcal R_{\mathrm{LA}}<1$, the compatibility condition
\eqref{eq:pf_compatibility} was scanned for sign changes along the
Perron--Frobenius surface \eqref{eq:pf_surface}; detected roots were refined by
\texttt{scipy.optimize.brentq} to absolute tolerance $10^{-11}$, and the
Jacobian eigenvalues at the reconstructed equilibria were then used to
identify coexistence of a stable radical-free equilibrium and a stable positive
equilibrium. A grid point was classified as bistable only when this coexistence
was detected.

The $50\times50$ resolution is used for visual mapping rather than for the
analytic persistence claim: the open neighbourhood of bistability around the
symmetric baseline follows from Remark~\ref{rem:asymm_persistence}, while the
grid displays how far that region extends in the chosen two-parameter slice.
Thus individual boundary pixels should be read qualitatively, not as a
high-precision bifurcation boundary.

The basin colouring in Figure~\ref{fig:bistability} was computed separately in
the invariant triangle \(\mathcal F\). Initial conditions on a uniform mesh were
integrated forward with \texttt{scipy.integrate.solve\_ivp}; each point was
assigned to the basin of \(E_0^{\mathrm{sym}}\) or \(E_+^{\mathrm{sym}}\) by its
terminal distance to the two attracting equilibria. The separatrix shown in the
phase portrait was obtained independently from the saddle
\(E_-^{\mathrm{sym}}\): starting at \(E_-^{\mathrm{sym}}\pm h e_s\), where
\(e_s\) is the stable eigenvector of the planar saddle and \(h=10^{-6}\), we
integrated the reversed planar vector field to trace the two branches of the
stable manifold. The resulting curve coincides with the boundary found by
forward basin classification.

This is precisely why the adiabatic theorem in the next section must be local
in state space. Since subcriticality of the frozen threshold does not eliminate
positive attractors, one cannot expect a global adiabatic statement in which
all trajectories follow the radical-free branch under slow parameter drift.
Instead, the result will control trajectories that start near that branch and
remain in a uniformly subcritical parameter region.

\section{Slow Drift and Local Adiabatic Tracking}\label{sec:adiabatic}

We now return to the nonautonomous model and ask whether the radical-free
centrist--alienated branch remains dynamically relevant when the effective
environment drifts slowly in time. Although this setting has an evident
slow-drift flavour, we do not invoke classical Fenichel theory directly.
Indeed, the slow evolution is prescribed at the level of the
effective parameter regime \(\theta(t)\) (and, in Section~\ref{sec:delay}, through the
externally given target regime \(\theta_\star(t)\)), rather than through a
standard autonomous fast--slow splitting of the state variables. For that
reason we work directly with a nonautonomous Lyapunov argument for the moving
branch \(\theta\mapsto x^\dagger(\theta)\).

Section~\ref{sec:symm} shows why the resulting statement must be formulated
locally in state space rather than globally over the whole simplex. In
particular, Example~\ref{ex:subcritical_coexistence} shows that even below the
local threshold there may coexist a radical-free attractor and a distinct
positive radicalised attractor. One therefore cannot expect an all-initial-data
adiabatic theorem for the radical-free branch. The natural statement is instead
local: if the effective parameters remain uniformly on the subcritical side of
the frozen threshold and vary slowly enough, then trajectories that start near
the moving radical-free branch remain close to it on the time interval under
consideration.

Let
\[
  x=(L,R,A)\in\mathcal T,
  \qquad
  \dot x = F(x,\theta),
\]
denote the vector field \eqref{eq:model} with frozen parameter value
\(\theta\in\mathcal P\). The radical-free branch from
Theorem~\ref{thm:frozen_threshold} is
\begin{equation}\label{eq:xdag_adiabatic}
  x^\dagger(\theta)
  :=
  \bigl(0,0,A^\dagger(\theta)\bigr),
  \qquad
  A^\dagger(\theta)=\frac{\eta(\theta)}{\eta(\theta)+\rho(\theta)},
  \qquad
  C^\dagger(\theta)=\frac{\rho(\theta)}{\eta(\theta)+\rho(\theta)}.
\end{equation}
For a fixed margin \(\nu\in(0,1)\), define the compact uniformly subcritical set
\begin{equation}\label{eq:Pminus_adiabatic}
  \mathcal P^-_\nu
  :=
  \{\theta\in\mathcal P:\ \mathcal R_{\mathrm{LA}}(\theta)\le 1-\nu\}.
\end{equation}
Thus every \(\theta\in\mathcal P^-_\nu\) lies a definite positive distance from
the frozen threshold surface
\[
  \Gamma
  :=
  \{\theta\in\mathcal P:\ \mathcal R_{\mathrm{LA}}(\theta)=1\}.
\]

\begin{proposition}[Uniform hyperbolicity and Lyapunov family]
\label{prop:uniform_hyperbolicity}
Fix \(\nu\in(0,1)\) and assume that \(\mathcal P^-_\nu\neq\varnothing\).
For \(\theta\in\mathcal P^-_\nu\), let
\[
  \mathcal J^\dagger(\theta):=
  D_xF\bigl(x^\dagger(\theta),\theta\bigr).
\]
Then there exists \(\omega_\nu>0\) such that
\begin{equation}\label{eq:uniform_hurwitz_adiabatic}
  s\!\bigl(\mathcal J^\dagger(\theta)\bigr)\le -\omega_\nu
  \qquad
  \text{for all }\theta\in\mathcal P^-_\nu.
\end{equation}
Moreover, there exist an open neighbourhood \(U_\nu\subset\mathbb R^{10}\) of
\(\mathcal P^-_\nu\) and a \(C^1\) map
\[
  P:U_\nu\to \mathbb R^{3\times3}
\]
such that \(P(\theta)\) is symmetric positive definite for every
\(\theta\in U_\nu\), and
\begin{equation}\label{eq:lyapunov_eq_adiabatic}
  \mathcal J^\dagger(\theta)^\top P(\theta)
  +
  P(\theta)\mathcal J^\dagger(\theta)
  =
  -I_3,
  \qquad
  \theta\in\mathcal P^-_\nu.
\end{equation}
Moreover, there exist constants \(m_\nu,M_\nu,L_\nu>0\) with
\begin{equation}\label{eq:P_bounds_adiabatic}
  m_\nu I_3\le P(\theta)\le M_\nu I_3,
  \qquad
  \|DP(\theta)\|\le L_\nu
  \qquad
  \text{for all }\theta\in\mathcal P^-_\nu.
\end{equation}
\end{proposition}

\begin{proof}
See \ref{app:prop_uniform_hyperbolicity}.
\end{proof}

The next theorem uses the Lyapunov family from
Proposition~\ref{prop:uniform_hyperbolicity} to show that trajectories
near the radical-free branch remain close to it under slow subcritical drift.

\begin{theorem}[Local adiabatic tracking away from the threshold surface]
\label{thm:adiabatic_tracking}
Fix \(\nu\in(0,1)\) and assume that \(\mathcal P^-_\nu\neq\varnothing\).
Then there exist constants
\[
  r_\nu>0, \qquad \varepsilon_\nu>0, \qquad \lambda_\nu>0, \qquad c_\nu>0
\]
(a neighbourhood radius, a slowness threshold, a decay rate, and a 
tracking constant, respectively) with the following property.

Let \(T>0\), and let
\[
  \theta:[0,T]\to\mathcal P^-_\nu
\]
be an absolutely continuous effective parameter path satisfying
\begin{equation}\label{eq:theta_slow_adiabatic}
  \operatorname*{ess\,sup}_{t\in[0,T]}\|\dot\theta(t)\|
  \le \varepsilon
  \qquad
  \text{for some }0<\varepsilon\le\varepsilon_\nu.
\end{equation}
Let \(x(\cdot)\) be the corresponding solution of
\begin{equation}\label{eq:nonaut_adiabatic}
  \dot x(t)=F(x(t),\theta(t)),
  \qquad
  x(0)=x_0\in\mathcal T.
\end{equation}
If
\begin{equation}\label{eq:initial_small_adiabatic}
  \|x_0-x^\dagger(\theta(0))\|\le \frac{r_\nu}{2c_\nu},
  \qquad
  \varepsilon\le \frac{r_\nu}{2c_\nu},
\end{equation}
then
\begin{equation}\label{eq:adiabatic_estimate}
  \|x(t)-x^\dagger(\theta(t))\|
  \le
  c_\nu e^{-\lambda_\nu t}\,
  \|x_0-x^\dagger(\theta(0))\|
  +
  c_\nu \varepsilon
  \qquad
  \text{for all }t\in[0,T].
\end{equation}
In particular,
\begin{equation}\label{eq:adiabatic_neighbourhood}
  \|x(t)-x^\dagger(\theta(t))\|\le r_\nu
  \qquad
  \text{for all }t\in[0,T].
\end{equation}
\end{theorem}

\begin{proof}
Set \(y(t):=x(t)-x^\dagger(\theta(t))\). Since
\(F(x^\dagger(\theta),\theta)=0\), the deviation satisfies
\[
  \dot y = \mathcal J^\dagger(\theta(t))\,y
           + \mathcal N(y(t),\theta(t))
           - \dot x^\dagger(\theta(t)),
\]
where \(\mathcal N\) is the quadratic remainder and
\(\|\dot x^\dagger\|\le B_\nu\varepsilon\). One then forms the
parameter-dependent Lyapunov function
\(V(t)=y(t)^\top P(\theta(t))y(t)\) using the family from
Proposition~\ref{prop:uniform_hyperbolicity}, applies Young's inequality and
Grönwall's lemma to obtain \eqref{eq:adiabatic_estimate} conditionally on
\(\|y(s)\|\le r_\nu\), and closes the argument by a bootstrap on the maximal
interval \([0,t_*]\) on which this bound holds: the condition
\eqref{eq:initial_small_adiabatic} forces \(t_*=T\).
Full details are given in \ref{app:thm_adiabatic}.
\end{proof}

\begin{remark}[Scope of local tracking and initially radicalised states]
\label{rem:adiabatic_scope_far_initial_data}
Theorem~\ref{thm:adiabatic_tracking} is deliberately not a global tracking
result. Its hypothesis \eqref{eq:initial_small_adiabatic} requires the initial
condition to lie in a neighbourhood of the moving radical-free branch
\(x^\dagger(\theta(0))\). If a society already has substantial radical support,
so that the initial state is far from \(x^\dagger(\theta(0))\), the theorem
makes no prediction about its subsequent evolution. In the bistable regime
identified in Section~\ref{sec:symm}, such an initial condition is governed by
basin membership relative to the separatrix between the radical-free and
positive attractors. It may converge to the radicalised equilibrium, or remain
near it under slow drift, even while the effective parameters stay locally
subcritical. Thus adiabatic tracking describes only the local basin of the
no-radical branch; initially radicalised states require an analysis of frozen
basin geometry and post-crossing state-space delay. This is the dynamical
content of the basin-boundary interpretation discussed in
Section~\ref{sec:discussion}.
\end{remark}

The relaxation law \eqref{eq:theta_relax} fits this framework whenever
\(\kappa_\theta\,\operatorname{diam}(\mathcal P)\le\varepsilon_\nu\).
The next section addresses threshold passage itself.

\section{Delayed Threshold Crossing under Implementation Lag}\label{sec:delay}

Theorem~\ref{thm:adiabatic_tracking} applies as long as the \emph{effective}
parameter path remains in a uniformly subcritical region of parameter space.
Once the effective regime approaches the frozen threshold surface
\[
  \Gamma=\{\theta\in\mathcal P:\ \mathcal R_{\mathrm{LA}}(\theta)=1\},
\]
the natural question is no longer the full state evolution on the simplex, but
the relative timing of threshold passage in the target and effective regimes.
This is precisely where implementation lag enters the model. In political
terms, target parameters describe the intended institutional environment,
whereas effective parameters describe the environment that actually governs
voter flows after administrative, legislative, and behavioural delay.

Given a target regime \(\theta_\star(t)\) and an effective regime \(\theta(t)\)
satisfying the relaxation law \eqref{eq:theta_relax}, define the corresponding
target and effective threshold trajectories by
\begin{equation}\label{eq:r_paths_delay}
  r_\star(t):=\mathcal R_{\mathrm{LA}}(\theta_\star(t)),
  \qquad
  r_{\mathrm{eff}}(t):=\mathcal R_{\mathrm{LA}}(\theta(t)).
\end{equation}
Thus \(r_\star(t)\) records the threshold status under instantaneous
implementation, and \(r_{\mathrm{eff}}(t)\) the status of the effective regime.
The threshold functional is regular on \(\mathcal P\):

\begin{lemma}[Regularity of the threshold functional]
\label{lem:RLA_regular}
The map
\[
  \mathcal R_{\mathrm{LA}}:\mathcal P\to(0,\infty)
\]
defined in \eqref{eq:RLA} is \(C^1\) on \(\mathcal P\). In particular, there
exists a finite constant \(L_{\mathcal R}>0\) such that
\begin{equation}\label{eq:R_lipschitz_delay}
  |\mathcal R_{\mathrm{LA}}(\theta_1)-\mathcal R_{\mathrm{LA}}(\theta_2)|
  \le
  L_{\mathcal R}\|\theta_1-\theta_2\|
  \qquad
  \text{for all }\theta_1,\theta_2\in\mathcal P.
\end{equation}
\end{lemma}

\begin{proof}
See \ref{app:lem_RLA_regular}.
\end{proof}

The main delay theorem below assumes a scalar ordering relation between
\(r_{\mathrm{eff}}\) and \(r_\star\). The next proposition gives a general
sufficient condition for that hypothesis.

\begin{proposition}[Scalar monotone path criterion]
\label{prop:scalar_lag_scalar}
Let \(J\subset\mathbb R\) be an interval, let \(\Theta:J\to\mathcal P\) be a
\(C^1\) path, and suppose that the scalar map
\[
  s\mapsto \mathcal R_{\mathrm{LA}}(\Theta(s))
\]
is nondecreasing on \(J\). Let \(I=[t_0,t_1]\), and assume that
\[
  \theta_\star(t)=\Theta(s_\star(t)),
  \qquad
  \theta(t)=\Theta(s(t)),
\]
where \(s_\star,s\in W^{1,\infty}(I;J)\) satisfy
\[
  \dot s(t)=\kappa_\theta\bigl(s_\star(t)-s(t)\bigr)
  \qquad
  \text{for a.e. }t\in I,
\]
with \(\kappa_\theta>0\), together with
\[
  \dot s_\star(t)\ge0 \quad\text{for a.e. }t\in I,
  \qquad
  s(t_0)\le s_\star(t_0).
\]
Then
\[
  s(t)\le s_\star(t)
  \qquad
  \text{for all }t\in I,
\]
and consequently
\[
  r_{\mathrm{eff}}(t)\le r_\star(t)
  \qquad
  \text{for all }t\in I.
\]
Thus the scalar lag condition \eqref{eq:scalar_lag_delay} holds with
\(\sigma=+1\). The reversed inequalities imply the corresponding conclusion for
\(\sigma=-1\).
\end{proposition}

\begin{proof}
Set \(e(t):=s(t)-s_\star(t)\). Then
\[
  \dot e(t)=-\kappa_\theta e(t)-\dot s_\star(t)\le -\kappa_\theta e(t)
\]
for a.e. \(t\in I\). Since \(e(t_0)\le0\), scalar comparison gives
\(e(t)\le0\) on \(I\). Hence \(s(t)\le s_\star(t)\), and the monotonicity of
\(s\mapsto \mathcal R_{\mathrm{LA}}(\Theta(s))\) yields
\[
  r_{\mathrm{eff}}(t)
  =
  \mathcal R_{\mathrm{LA}}(\Theta(s(t)))
  \le
  \mathcal R_{\mathrm{LA}}(\Theta(s_\star(t)))
  =
  r_\star(t).
\]
\end{proof}

We now turn to delayed threshold passage itself. The sign parameter
\[
  \sigma\in\{+1,-1\}
\]
encodes the direction of crossing:
\[
  \sigma=+1
  \quad\text{for upward passage }(r_\star:<1\to>1),
\]
\[
  \sigma=-1
  \quad\text{for downward passage }(r_\star:>1\to<1).
\]

\begin{theorem}[Delayed threshold passage of the effective regime]
\label{thm:delay}
Let \(I=[t_0,t_1]\) be a compact interval, let
\[
  \theta_\star,\theta\in W^{1,\infty}(I;\mathcal P),
\]
and assume that \(\theta\) satisfies \eqref{eq:theta_relax} with some
\(\kappa_\theta>0\). Fix \(\sigma\in\{+1,-1\}\) and suppose that:

\begin{enumerate}[label=\textup{(D\arabic*)},leftmargin=3.2em]
  \item \label{it:D1_delay_new}
  there exists a unique time \(t_\star\in(t_0,t_1)\) such that
  \begin{equation}\label{eq:signed_crossing_delay}
    \sigma\bigl(r_\star(t)-1\bigr)<0
    \quad \text{for } t\in[t_0,t_\star),
    \qquad
    r_\star(t_\star)=1;
  \end{equation}

  \item \label{it:D2_delay_new}
  there exist constants \(m_\star>0\) and \(\Delta_\star>0\) such that
  \(t_\star+\Delta_\star\le t_1\) and
  \begin{equation}\label{eq:transversal_delay}
    \sigma\bigl(r_\star(t)-1\bigr)
    \ge
    m_\star (t-t_\star)
    \qquad
    \text{for all } t\in[t_\star,t_\star+\Delta_\star];
  \end{equation}

  \item \label{it:D3_delay_new}
  the scalar lag condition
  \begin{equation}\label{eq:scalar_lag_delay}
    \sigma\bigl(r_{\mathrm{eff}}(t)-1\bigr)
    \le
    \sigma\bigl(r_\star(t)-1\bigr)
    \qquad
    \text{for all } t\in I
  \end{equation}
  holds;

  \item \label{it:D4_delay_new}
  there exists \(V_\star<\infty\) such that
  \begin{equation}\label{eq:target_speed_delay}
    \|\dot\theta_\star(t)\|
    \le
    V_\star
    \qquad
    \text{for a.e. }t\in I.
  \end{equation}
\end{enumerate}

Define
\begin{equation}\label{eq:E_star_delay}
  E_\star
  :=
  \|\theta(t_0)-\theta_\star(t_0)\|
  +
  \frac{V_\star}{\kappa_\theta},
\end{equation}
and assume that
\begin{equation}\label{eq:delay_smallness}
  \frac{L_{\mathcal R}E_\star}{m_\star}\le \Delta_\star.
\end{equation}
Then the effective crossing time
\begin{equation}\label{eq:teff_sigma}
  t_{\mathrm{eff}}^\sigma
  :=
  \inf\Bigl\{
    t\in I:\ \sigma\bigl(r_{\mathrm{eff}}(t)-1\bigr)\ge 0
  \Bigr\}
\end{equation}
is well defined and satisfies
\begin{equation}\label{eq:delay_bound}
  0\le t_{\mathrm{eff}}^\sigma-t_\star
  \le
  \frac{L_{\mathcal R}E_\star}{m_\star}.
\end{equation}
In particular, the effective regime cannot cross the frozen threshold before
the target regime, and it must cross within a delay window of length
\(L_{\mathcal R}E_\star/m_\star\) after target passage.
\end{theorem}

\begin{proof}
See \ref{app:thm_delay}.
\end{proof}

The novelty of Theorem~\ref{thm:delay} is model-specific rather than
methodological: the frozen Perron--Frobenius threshold becomes an explicit
implementation-lag bound for the target/effective split.

The estimate becomes especially transparent when the target and effective
regimes agree at the beginning of the crossing episode.

\begin{corollary}[$O(\kappa_\theta^{-1})$ delay under initial matching]
\label{cor:delay_matching}
Under the assumptions of Theorem~\ref{thm:delay}, if in addition
\begin{equation}\label{eq:matching_delay}
  \theta(t_0)=\theta_\star(t_0),
\end{equation}
then
\begin{equation}\label{eq:delay_bound_matching}
  0\le t_{\mathrm{eff}}^\sigma-t_\star
  \le
  \frac{L_{\mathcal R}}{m_\star}\,\frac{V_\star}{\kappa_\theta}.
\end{equation}
\end{corollary}

\begin{proof}
Under \eqref{eq:matching_delay}, $E_\star=V_\star/\kappa_\theta$;
substitute into \eqref{eq:delay_bound}.
\end{proof}

Theorem~\ref{thm:delay} concerns threshold passage in parameter space only. It
does not quantify the additional time the state trajectory may remain near the
radical-free branch after \(t_{\mathrm{eff}}^\sigma\). The following example
illustrates both the sharp parameter-space delay and the further post-crossing
state lag.

\begin{example}[Monotone ramp and nearly sharp delay bound]
\label{ex:delay_ramp}
To visualize the implementation-lag mechanism, we consider an upward
crossing (\(\sigma=+1\)) in the symmetric reduction \eqref{eq:model_symm},
varying only the aggregate recruitment parameter. Let
\[
  \mu=1,\qquad \delta=0.3,\qquad \eta=0.2,\qquad \rho=0.1,
\]
and prescribe the target ramp
\[
  \beta_\star(t)=2+0.012\,t,
  \qquad
  \dot\beta=\kappa_\theta(\beta_\star(t)-\beta),
  \qquad
  \kappa_\theta=0.2,
\]
with initial matching
\[
  \beta(0)=\beta_\star(0)=2.
\]
Then
\[
  \beta(t)=2+0.012\,t-0.06\bigl(1-e^{-0.2t}\bigr),
\]
and
\[
  r_\star(t)=\mathcal R_{\mathrm{LA}}^{\mathrm{sym}}(\beta_\star(t))
  =\frac{\beta_\star(t)}{3}+0.2,
  \qquad
  r_{\mathrm{eff}}(t)=\mathcal R_{\mathrm{LA}}^{\mathrm{sym}}(\beta(t))
  =\frac{\beta(t)}{3}+0.2.
\]
Here the target and effective regimes lie on the one-parameter path
\(\Theta(\beta)\) obtained by varying \(\beta\) while keeping
\((\mu,\delta,\eta,\rho)\) fixed. Since
\(\mathcal R_{\mathrm{LA}}^{\mathrm{sym}}(\beta)\) is affine increasing in
\(\beta\), Proposition~\ref{prop:scalar_lag_scalar} implies that the scalar
lag condition \eqref{eq:scalar_lag_delay} holds automatically with
\(\sigma=+1\).

The target threshold is crossed at
\[
  t_\star=\frac{2.4-2}{0.012}=\frac{100}{3}\approx 33.33,
\]
whereas the effective threshold is crossed later at
\[
  t_{\mathrm{eff}}\approx 38.33.
\]
In this one-parameter affine setting the slope of
\(\mathcal R_{\mathrm{LA}}^{\mathrm{sym}}\) with respect to \(\beta\)
gives the Lipschitz constant directly:
\[
  L_{\mathcal R}=\frac{\rho}{\mu(\eta+\rho)}=\frac13,
  \qquad
  V_\star=0.012,
  \qquad
  m_\star=\dot r_\star(t)=0.004.
\]
Corollary~\ref{cor:delay_matching} then yields the explicit bound
\[
  0\le t_{\mathrm{eff}}-t_\star
  \le
  \frac{L_{\mathcal R}}{m_\star}\,\frac{V_\star}{\kappa_\theta}
  =5.
\]
The observed delay is
\[
  t_{\mathrm{eff}}-t_\star\approx 4.998,
\]
so the bound is nearly attained in this monotone ramp example.

To connect the parameter-space statement with the state dynamics, we solve the
symmetric system \eqref{eq:model_symm} with initial condition
\[
  P(0)=10^{-3},
  \qquad
  A(0)=\frac{\eta}{\eta+\rho}=\frac23.
\]
For a visibility threshold \(\xi \ge P(0)\), let
\(t_{\mathrm{obs}}(\xi) := \inf\{t\ge t_{\mathrm{eff}}:\ P(t)\ge \xi\}\)
denote the first time after effective threshold crossing at which the
radical share reaches level \(\xi\). The baseline below uses \(\xi=P(0)\) as the visibility criterion;
Table~\ref{tab:obs_delay} shows that the qualitative decomposition is
robust across larger values of \(\xi\).

The numerical solution stays close to the radical-free branch throughout the
parameter-delay window and develops visible radical growth only later.
Figure~\ref{fig:delay} shows the separation between target and effective
threshold crossing for \(\kappa_\theta=0.2\), and
Table~\ref{tab:delay_decomposition} reports the decomposition of the total lag
into the parameter-space part \(t_{\mathrm{eff}}-t_\star\) and the additional
state-space part \(t_{\mathrm{obs}}-t_{\mathrm{eff}}\). In this example the
parameter delay follows the sharp \(O(\kappa_\theta^{-1})\) prediction,
whereas the state-space delay is substantially larger and varies much more
weakly with \(\kappa_\theta\).

\begin{table}[t]
\centering
\caption{Decomposition of parameter-space and state-space delay in
Example~\ref{ex:delay_ramp}.}
\label{tab:delay_decomposition}
\begin{tabular}{@{}ccccc@{}}
\toprule
\(\kappa_\theta\) & \(t_\star\) & \(t_{\mathrm{eff}}\) &
\(t_{\mathrm{eff}}-t_\star\) & \(t_{\mathrm{obs}}-t_{\mathrm{eff}}\) \\
\midrule
0.10 & 33.33 & 43.20 & 9.87 & 41.1 \\
0.20 & 33.33 & 38.33 & 4.998 & 37.7 \\
0.50 & 33.33 & 35.33 & 2.00 & 35.2 \\
\bottomrule
\end{tabular}
\end{table}

Table~\ref{tab:obs_delay} shows that this qualitative decomposition is
robust to the choice of visibility threshold $\xi$. The parameter-space
delay $t_{\mathrm{eff}}-t_\star$ retains its sharp $O(\kappa_\theta^{-1})$
scaling across all four criteria, while the state-space delay
$t_{\mathrm{obs}}(\xi)-t_{\mathrm{eff}}$ remains substantially larger and
varies much more weakly with $\kappa_\theta$ regardless of $\xi$.

\begin{table}[t]
\centering
\caption{Decomposition of threshold delay across alternative visibility
criteria. For each implementation speed $\kappa_\theta$ and each
observable level $\xi$, the parameter-space delay
$t_{\mathrm{eff}}-t_\star$ (proved to scale as
$O(\kappa_\theta^{-1})$) and the state-space delay
$t_{\mathrm{obs}}(\xi)-t_{\mathrm{eff}}$ are reported. The seed level
is $P(0)=10^{-3}$; $t_\star\approx 33.3$ throughout.}
\label{tab:obs_delay}
\begin{tabular}{@{}ccccccc@{}}
\toprule
& & & \multicolumn{4}{c}{$t_{\mathrm{obs}}(\xi)-t_{\mathrm{eff}}$} \\
\cmidrule(l){4-7}
$\kappa_\theta$ & $t_{\mathrm{eff}}$ & $t_{\mathrm{eff}}-t_\star$
  & $\xi=P(0)$ & $\xi=2P(0)$ & $\xi=0.01$ & $\xi=0.05$ \\
\midrule
0.10 & 43.20 & 9.87 & 41.1 & 45.2 & 53.7 & 62.3 \\
0.20 & 38.33 & 5.00 & 37.7 & 42.1 & 51.1 & 60.0 \\
0.50 & 35.33 & 2.00 & 35.2 & 39.9 & 49.3 & 58.5 \\
\bottomrule
\end{tabular}
\end{table}

Figure~\ref{fig:delay_scaling} plots \(t_{\mathrm{eff}}-t_\star\) against
\(1/\kappa_\theta\) for seven implementation speeds together with the
theoretical line \(1/\kappa_\theta\) from
Corollary~\ref{cor:delay_matching}, showing close agreement throughout the
tested range. Numerical trajectories were computed with
\texttt{scipy.integrate.solve\_ivp} \cite{scipy2020}.
\end{example}

\begin{figure}[t]
\centering
\includegraphics[width=0.60\textwidth]{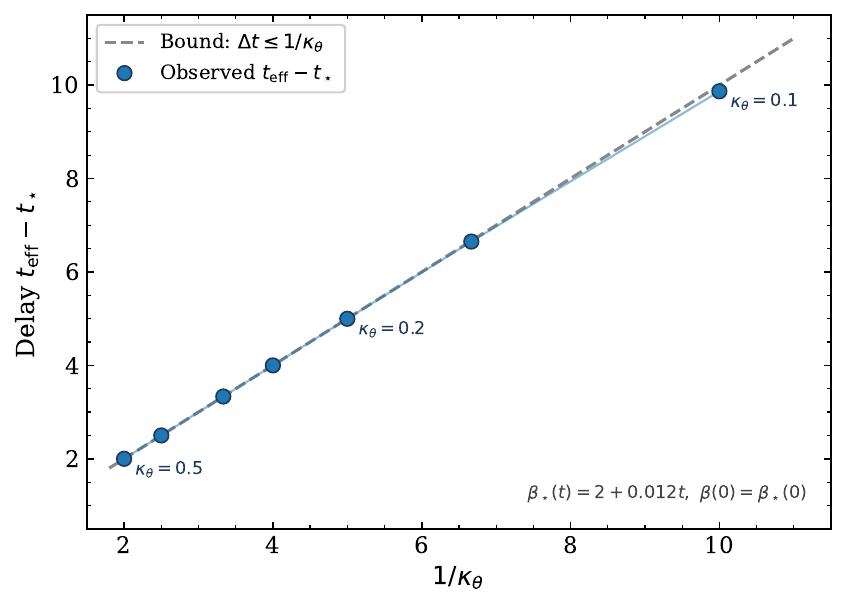}
\caption{Parameter-space delay vs.\ implementation speed,
Example~\ref{ex:delay_ramp}. Circles: observed \(t_{\mathrm{eff}}-t_\star\);
dashed: bound \(1/\kappa_\theta\) from Corollary~\ref{cor:delay_matching}.}
\label{fig:delay_scaling}
\end{figure}

\begin{figure}[t]
\centering
\begin{subfigure}[t]{0.48\textwidth}
  \centering
  \includegraphics[width=\textwidth]{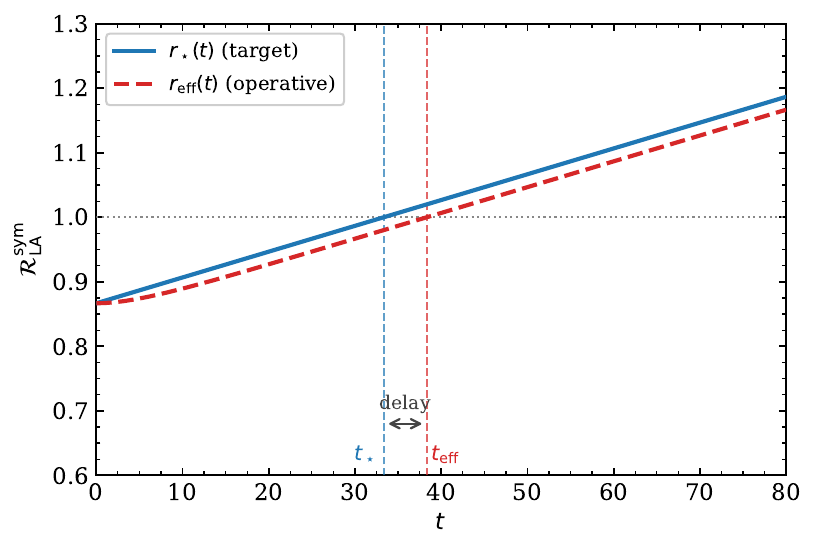}
  \caption{Target and effective threshold trajectories.}
\end{subfigure}\hfill
\begin{subfigure}[t]{0.48\textwidth}
  \centering
  \includegraphics[width=\textwidth]{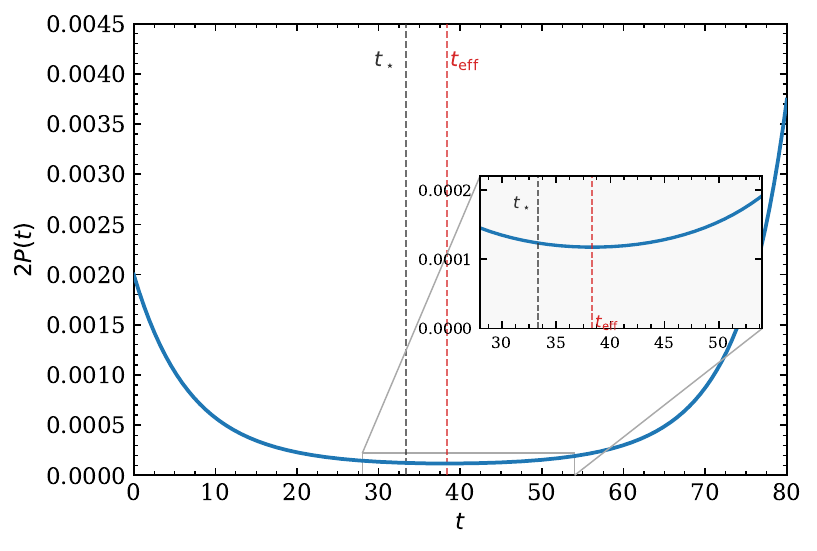}
  \caption{State response after effective threshold passage.}
\end{subfigure}
\caption{Implementation lag in Example~\ref{ex:delay_ramp}
(\(\kappa_\theta=0.2\)). Vertical lines: \(t_\star\) and \(t_{\mathrm{eff}}\).}
\label{fig:delay}
\end{figure}

\begin{figure}[t]
\centering
\includegraphics[width=0.72\textwidth]{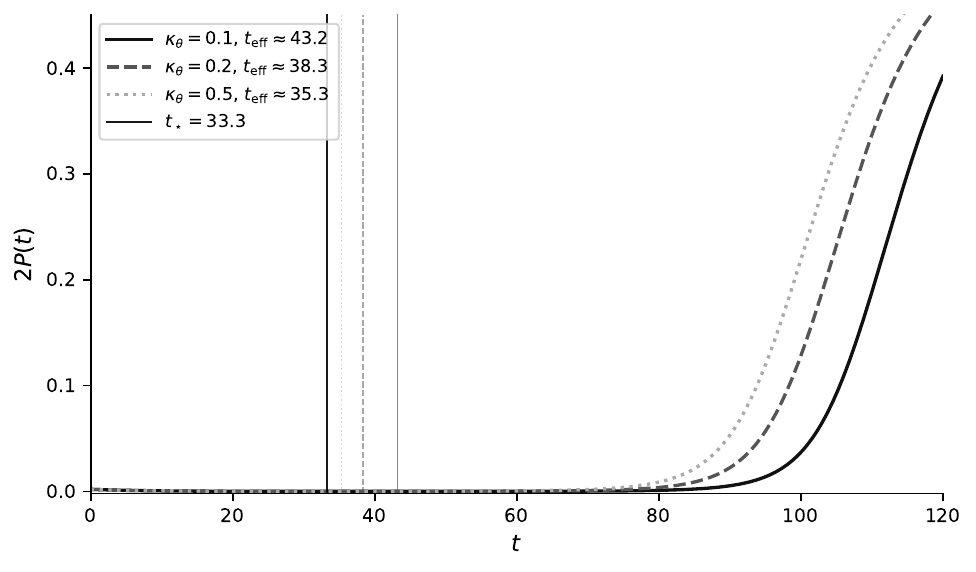}
\caption{Radical share \(2P(t)\) for \(\kappa_\theta\in\{0.1,0.2,0.5\}\),
Example~\ref{ex:delay_ramp}. Black vertical line: \(t_\star\approx 33.3\);
grey lines: \(t_{\mathrm{eff}}\).}
\label{fig:delay_state_compare}
\end{figure}

In political terms, the theorem and example describe a situation in which a
society has already crossed the \emph{target} threshold for radicalisation,
while the \emph{effective} institutional environment remains temporarily on the
old side of the instability boundary. This creates a mathematically precise lag window between intended
structural change and effective voter dynamics.

\begin{remark}[Piecewise-smooth target paths]
Theorem~\ref{thm:delay} extends to piecewise-smooth \(\theta_\star\) by
restarting the estimate at each jump discontinuity \(t_j\) with initial
mismatch \(\|\theta(t_j^+)-\theta_\star(t_j^+)\|\).
\end{remark}

\FloatBarrier
\section{Discussion}\label{sec:discussion}

The main substantive conclusion of the paper is that the frozen threshold is
preventive rather than curative. Persistent alienation can create
below-threshold coexistence of a radical-free centrist--alienated equilibrium
and a positive radicalised equilibrium, separated by a saddle born at a
saddle-node bifurcation. Returning to \(\mathcal R_{\mathrm{LA}}<1\) therefore
restores only local stability of the radical-free branch; it does not
guarantee recovery once the state has entered the basin of the radicalised
attractor.

Equivalently, the stable manifold of the saddle acts as a social tipping
boundary. A policy environment may be subcritical in the local
Perron--Frobenius sense and still contain a finite-amplitude radicalised
attractor. If a shock, a long period of weak implementation, or gradual
accumulation of alienation pushes the state across this basin boundary,
radical support can remain self-sustaining even after the original policy
environment is restored. The warning for policymakers is therefore temporal:
restoring the local threshold late is not equivalent to preventing threshold
loss early. Effective intervention must either keep trajectories on the
radical-free side of the basin boundary or alter the basin geometry itself,
for example by reducing chronic alienation and the mobilisation rate from the
alienation reservoir.

Delayed implementation adds a second source of inertia. Because the political
environment is split into a target regime \(\theta_\star(t)\) and an effective
regime \(\theta(t)\), announced threshold restoration need not coincide with
operative threshold restoration. When a visible post-crossing state response exists, the three times
\(t_\star < t_{\mathrm{eff}} \le t_{\mathrm{obs}}\) separate for
distinct reasons: the first gap \(t_{\mathrm{eff}}-t_\star\) is a
parameter-space delay generated by finite implementation speed, whereas 
the second gap \(t_{\mathrm{obs}}-t_{\mathrm{eff}}\) is a state-space delay
generated by basin geometry and initial condition. In the monotone ramp
example the first lag follows the sharp \(O(\kappa_\theta^{-1})\) prediction
almost exactly, whereas the second is substantially larger and much less
sensitive to \(\kappa_\theta\). In scenarios where the state has entered the basin of the radicalised
attractor before threshold restoration, \(t_{\mathrm{obs}}\) need not
be finite.

These two effects are the substantive payoff of the analysis. The
Perron--Frobenius threshold, frozen equilibrium geometry, and local adiabatic
theorem are the tools that make them precise. They are not claimed as
stand-alone methodological innovations; their role is to establish rigorously
that delayed institutional implementation and persistent alienation interact to
produce hysteresis-type path dependence and a two-stage delay in radical
support dynamics.

The main limitations are equally clear. The adiabatic theorem is local
in state space, and the delay theorem bounds threshold passage in
parameter space rather than the full post-crossing escape. The model is
structurally transparent rather than calibrated; its value is the
isolation of a mechanism, not a quantitative forecast. Natural extensions
include asymmetric or state-dependent implementation speeds, feedback
from the state to the target regime, stochastic forcing, and a rigorous
analysis of post-crossing state-space delay. On the empirical side, the
parameter/proxy correspondence of Table~\ref{tab:proxies} suggests that
the \(O(\kappa_\theta^{-1})\) bound is in principle testable: if
\(\kappa_\theta\) is proxied by the inverse of the legislative-to-effect
lag, the bound suggests an empirically interpretable hypothesis about
how long radical growth remains suppressed after a reform target is
announced, conditional on future calibration of the proxy mapping.
The German case discussed in Section~\ref{sec:intro} provides one
instance where both the target date and the implementation lag are
independently documented \cite{bundesrechnungshof2021}.


\FloatBarrier
\appendix
\section{Auxiliary proofs}

\subsection{Proof of Proposition~\ref{prop:forward_invariant}}
\label{app:prop_forward_invariant}

\begin{proof}

Local well posedness follows because \(\theta(\cdot)\) is absolutely continuous,
\(\mathcal P\) is compact, and the vector field in \eqref{eq:model} is
continuous in \(t\) and locally Lipschitz in \((L,R,A)\).

It remains to prove forward invariance of \(\mathcal T\). On the faces
\(L=0\), \(R=0\), and \(A=0\),
\[
  \dot L\big|_{L=0}=\gamma_{RL}RC\ge0,\qquad
  \dot R\big|_{R=0}=\gamma_{LR}LC\ge0,\qquad
  \dot A\big|_{A=0}=\eta C\ge0.
\]
On the face \(C=0\),
\[
  \dot C=-\mathbf 1^\top\dot v-\dot A
  =\mu_L L+\mu_R R+\rho A\ge0,
\]
because the mobilisation terms cancel. Thus the vector field points inward on
every boundary face of \(\mathcal T\), so \(\mathcal T\) is forward invariant.
Since \(\mathcal T\) is compact and positively invariant, no finite-time
blow-up can occur, and the local solution extends globally.

\end{proof}

\subsection{Proof of Proposition~\ref{prop:uniform_hyperbolicity}}
\label{app:prop_uniform_hyperbolicity}

\begin{proof}

By Theorem~\ref{thm:frozen_threshold}, every \(\theta\in\mathcal P^-_\nu\)
makes the frozen radical-free equilibrium \(x^\dagger(\theta)\) locally
asymptotically stable. Hence
\[
  s\!\bigl(\mathcal J^\dagger(\theta)\bigr)<0
  \qquad
  \text{for all }\theta\in\mathcal P^-_\nu.
\]
The map \(\theta\mapsto \mathcal J^\dagger(\theta)\) is continuous because
\(F\) is smooth in both variables and \(x^\dagger(\theta)\) is a smooth
function of \((\eta,\rho)\). Therefore
\[
  \theta\mapsto s\!\bigl(\mathcal J^\dagger(\theta)\bigr)
\]
is continuous on the compact set \(\mathcal P^-_\nu\), so
\[
  -\omega_\nu
  :=
  \max_{\theta\in\mathcal P^-_\nu}
  s\!\bigl(\mathcal J^\dagger(\theta)\bigr)
  <0,
\]
which proves \eqref{eq:uniform_hurwitz_adiabatic}.

For each fixed Hurwitz matrix \(\mathcal J^\dagger(\theta)\), the Lyapunov
equation \eqref{eq:lyapunov_eq_adiabatic} has the unique symmetric positive
definite solution
\begin{equation}\label{eq:P_integral_adiabatic}
  P(\theta)
  =
  \int_0^\infty
    e^{\mathcal J^\dagger(\theta)^\top s}
    e^{\mathcal J^\dagger(\theta)s}\,ds.
\end{equation}
By \eqref{eq:uniform_hurwitz_adiabatic} and continuity of
\(\theta\mapsto \mathcal J^\dagger(\theta)\), there exist constants
\(C_\nu,\omega'_\nu>0\) such that
\[
  \bigl\|e^{\mathcal J^\dagger(\theta)s}\bigr\|
  \le
  C_\nu e^{-\omega'_\nu s}
  \qquad
  \text{for all }\theta\in\mathcal P^-_\nu,\ s\ge0.
\]
Hence the integral \eqref{eq:P_integral_adiabatic} converges uniformly on
\(\mathcal P^-_\nu\). Since \(\mathcal J^\dagger\) is \(C^1\) on the ambient
parameter set \(\mathcal P\), differentiation under the integral sign shows
that \(P\) is the restriction to \(\mathcal P^-_\nu\) of a \(C^1\) map defined
on a neighbourhood \(U_\nu\) of \(\mathcal P^-_\nu\).

The bounds \eqref{eq:P_bounds_adiabatic} now follow from continuity and
compactness. Since each \(P(\theta)\) is positive definite, the minimal and
maximal eigenvalue functions attain positive finite extrema on
\(\mathcal P^-_\nu\), which gives \(m_\nu\) and \(M_\nu\); the derivative
bound follows from compactness of \(\mathcal P^-_\nu\).

\end{proof}

\subsection{Proof of Theorem~\ref{thm:adiabatic_tracking}}
\label{app:thm_adiabatic}

\begin{proof}

We divide the argument into four steps.

\smallskip
\noindent\emph{Step 1: equation for the deviation from the moving branch.}
Define
\[
  y(t):=x(t)-x^\dagger(\theta(t)).
\]
Because \(F(x^\dagger(\theta),\theta)=0\) for every
\(\theta\in\mathcal P^-_\nu\), we obtain for almost every \(t\in[0,T]\)
\begin{equation}\label{eq:y_eq_adiabatic}
  \dot y
  =
  \mathcal J^\dagger(\theta(t))\,y
  +
  \mathcal N\bigl(y(t),\theta(t)\bigr)
  -
  \dot x^\dagger(\theta(t)),
\end{equation}
where
\begin{equation}\label{eq:N_def_adiabatic}
  \mathcal N(y,\theta)
  :=
  F\bigl(x^\dagger(\theta)+y,\theta\bigr)
  -
  F\bigl(x^\dagger(\theta),\theta\bigr)
  -
  \mathcal J^\dagger(\theta)y.
\end{equation}
Since \(F\) is \(C^2\) on the compact set \(\mathcal T\times\mathcal P\),
there exist constants \(r_0>0\) and \(K_0<\infty\) such that
\begin{equation}\label{eq:remainder_bound_adiabatic}
  \|\mathcal N(y,\theta)\|\le K_0\|y\|^2
\end{equation}
whenever \(\theta\in\mathcal P^-_\nu\), \(\|y\|\le r_0\), and
\(x^\dagger(\theta)+y\in\mathcal T\).

Moreover, the map \(\theta\mapsto x^\dagger(\theta)\) is \(C^1\) on
\(\mathcal P\), because
\[
  x^\dagger(\theta)=\left(0,0,\frac{\eta(\theta)}{\eta(\theta)+\rho(\theta)}\right)
\]
and the denominator \(\eta(\theta)+\rho(\theta)\) is uniformly bounded away
from zero on \(\mathcal P\). Hence there exists a
constant \(B_\nu<\infty\) such that
\begin{equation}\label{eq:xdag_dot_adiabatic}
  \|\dot x^\dagger(\theta(t))\|
  \le
  B_\nu\|\dot\theta(t)\|
  \le
  B_\nu\varepsilon
  \qquad
  \text{for a.e. }t\in[0,T].
\end{equation}

\smallskip
\noindent\emph{Step 2: a parameter-dependent Lyapunov function.}
Let \(P(\theta)\) be the positive definite solution of
\eqref{eq:lyapunov_eq_adiabatic} from
Proposition~\ref{prop:uniform_hyperbolicity}, and define
\[
  V(t):=y(t)^\top P(\theta(t))y(t).
\]
By \eqref{eq:P_bounds_adiabatic},
\begin{equation}\label{eq:V_equiv_adiabatic}
  m_\nu\|y\|^2\le V\le M_\nu\|y\|^2.
\end{equation}
Since \(P\) is \(C^1\) and \(\theta\) is absolutely continuous, \(V\) is
absolutely continuous as well. For almost every \(t\in[0,T]\),
\begin{align}
  \dot V
  &=
  y^\top\bigl(\mathcal J^\dagger(\theta)^\top P(\theta)
  +P(\theta)\mathcal J^\dagger(\theta)\bigr)y
  +2y^\top P(\theta)\mathcal N(y,\theta)
  \notag\\
  &\qquad
  -2y^\top P(\theta)\dot x^\dagger(\theta)
  +y^\top \dot P(\theta)y.
  \label{eq:V_dot_raw_adiabatic}
\end{align}
Using \eqref{eq:lyapunov_eq_adiabatic},
\eqref{eq:remainder_bound_adiabatic},
\eqref{eq:xdag_dot_adiabatic}, and
\[
  \|\dot P(\theta(t))\|
  \le
  \|DP(\theta(t))\|\,\|\dot\theta(t)\|
  \le
  L_\nu \varepsilon,
\]
we obtain, whenever \(\|y(t)\|\le r_0\),
\begin{equation}\label{eq:V_dot_est1_adiabatic}
  \dot V
  \le
  -\|y\|^2
  +2M_\nu K_0\|y\|^3
  +2M_\nu B_\nu\varepsilon\|y\|
  +L_\nu\varepsilon\|y\|^2.
\end{equation}

\smallskip
\noindent\emph{Step 3: closing the differential inequality.}
Choose \(r_\nu\in(0,r_0]\) and \(\varepsilon_\nu\in(0,1]\) such that
\begin{equation}\label{eq:choice_r_eps_adiabatic}
  2M_\nu K_0 r_\nu + L_\nu\varepsilon_\nu \le \frac12.
\end{equation}
Then for every \(0<\varepsilon\le\varepsilon_\nu\) and every time \(t\) such
that \(\|y(t)\|\le r_\nu\), estimate \eqref{eq:V_dot_est1_adiabatic} gives
\begin{equation}\label{eq:V_dot_est2_adiabatic}
  \dot V
  \le
  -\frac12\|y\|^2
  +2M_\nu B_\nu\varepsilon\|y\|.
\end{equation}
Applying Young's inequality,
\[
  2M_\nu B_\nu\varepsilon\|y\|
  \le
  \frac14\|y\|^2+4M_\nu^2B_\nu^2\varepsilon^2,
\]
we obtain
\begin{equation}\label{eq:V_dot_est3_adiabatic}
  \dot V
  \le
  -\frac14\|y\|^2
  +4M_\nu^2B_\nu^2\varepsilon^2
  \le
  -\frac{1}{4M_\nu}V
  +4M_\nu^2B_\nu^2\varepsilon^2.
\end{equation}
Set
\begin{equation}\label{eq:lambda_adiabatic}
  \lambda_\nu:=\frac{1}{8M_\nu}.
\end{equation}
Then \eqref{eq:V_dot_est3_adiabatic} implies
\begin{equation}\label{eq:V_gronwall_adiabatic}
  \dot V
  \le
  -2\lambda_\nu V+4M_\nu^2B_\nu^2\varepsilon^2.
\end{equation}
By Gr\"onwall's inequality,
\begin{equation}\label{eq:V_final_adiabatic}
  V(t)
  \le
  e^{-2\lambda_\nu t}V(0)
  +
  \frac{4M_\nu^2B_\nu^2}{2\lambda_\nu}\varepsilon^2
  \qquad
  \text{for all }t\in[0,T]
\end{equation}
as long as \(\|y(s)\|\le r_\nu\) on \([0,t]\). Using
\eqref{eq:V_equiv_adiabatic}, we get
\begin{equation}\label{eq:y_est_intermediate_adiabatic}
  \|y(t)\|
  \le
  \sqrt{\frac{M_\nu}{m_\nu}}
  e^{-\lambda_\nu t}\|y(0)\|
  +
  \sqrt{\frac{2M_\nu^2B_\nu^2}{m_\nu\lambda_\nu}}\,\varepsilon.
\end{equation}
Define
\begin{equation}\label{eq:c_nu_adiabatic}
  c_\nu
  :=
  \max\left\{
    \sqrt{\frac{M_\nu}{m_\nu}},
    \sqrt{\frac{2M_\nu^2B_\nu^2}{m_\nu\lambda_\nu}}
  \right\}.
\end{equation}
Then
\begin{equation}\label{eq:y_est_clean_adiabatic}
  \|y(t)\|
  \le
  c_\nu e^{-\lambda_\nu t}\|y(0)\|+c_\nu\varepsilon
\end{equation}
as long as \(\|y(s)\|\le r_\nu\) on \([0,t]\).

\smallskip
\noindent\emph{Step 4: bootstrap and conclusion.}
By Proposition~\ref{prop:forward_invariant}, the solution of
\eqref{eq:nonaut_adiabatic} exists on all of \([0,T]\). Let
\[
  t_*:=\sup\{t\in[0,T]:\ \|y(s)\|\le r_\nu\text{ for all }s\in[0,t]\}.
\]
Since \eqref{eq:initial_small_adiabatic} holds, one has \(t_*>0\). If
\(t_*<T\), then continuity of \(y\) implies \(\|y(t_*)\|=r_\nu\). But applying \eqref{eq:y_est_clean_adiabatic} at time \(t=t_*\) gives
\[
  \|y(t_*)\|
  \le
  c_\nu e^{-\lambda_\nu t_*}\|y(0)\|+c_\nu\varepsilon
  \le
  e^{-\lambda_\nu t_*}\frac{r_\nu}{2}+\frac{r_\nu}{2}
  < r_\nu,
\]
because \(t_*>0\) implies \(e^{-\lambda_\nu t_*}<1\). This contradicts
\(\|y(t_*)\|=r_\nu\), and therefore \(t_*=T\). Hence
\eqref{eq:y_est_clean_adiabatic} holds on all of \([0,T]\), which is exactly
\eqref{eq:adiabatic_estimate}. The bound
\eqref{eq:adiabatic_neighbourhood} follows immediately from
\eqref{eq:initial_small_adiabatic}.

\end{proof}

\subsection{Proof of Lemma~\ref{lem:RLA_regular}}
\label{app:lem_RLA_regular}

\begin{proof}

By \eqref{eq:RLA},
\[
  \mathcal R_{\mathrm{LA}}(\theta)
  =
  \lambda_{\mathrm{PF}}\!\left(
    M(\theta)^{-1}
    \bigl[
      C^\dagger(\theta)K(\theta)+A^\dagger(\theta)D(\theta)
    \bigr]
  \right),
\]
where
\[
  A^\dagger(\theta)=\frac{\eta(\theta)}{\eta(\theta)+\rho(\theta)},
  \qquad
  C^\dagger(\theta)=\frac{\rho(\theta)}{\eta(\theta)+\rho(\theta)}.
\]
Since \(\eta(\theta)+\rho(\theta)\) is uniformly bounded away from zero on the
compact set \(\mathcal P\), the maps
\(\theta\mapsto A^\dagger(\theta)\) and \(\theta\mapsto C^\dagger(\theta)\)
are \(C^1\). Therefore the matrix
\[
  B^\dagger(\theta):=
  M(\theta)^{-1}
  \bigl[
    C^\dagger(\theta)K(\theta)+A^\dagger(\theta)D(\theta)
  \bigr]
\]
depends \(C^1\)-smoothly on \(\theta\). For each \(\theta\in\mathcal P\),
\(B^\dagger(\theta)\) is a positive \(2\times2\) matrix, so its
Perron--Frobenius eigenvalue is algebraically simple. Since
\(B^\dagger(\theta)\) depends \(C^1\)-smoothly on \(\theta\), standard
perturbation theory for simple eigenvalues implies that
\[
  \theta\mapsto \lambda_{\mathrm{PF}}(B^\dagger(\theta))
\]
is \(C^1\) on \(\mathcal P\); see, for example, Kato
\cite{Kato1995}. Since \(\mathcal P\) is compact, every \(C^1\)
map on \(\mathcal P\) is Lipschitz, which proves
\eqref{eq:R_lipschitz_delay}.

\end{proof}

\subsection{Proof of Proposition~\ref{prop:symm_benchmark}}
\label{app:prop_symm_benchmark}

\begin{proof}
Linearising \eqref{eq:model_symm} at \eqref{eq:E0_symm} gives the triangular
Jacobian
\[
  J\bigl(E_0^{\mathrm{sym}}\bigr)
  =
  \begin{pmatrix}
    (\beta-\mu)+(\delta-\beta)A^\dagger & 0\\[3pt]
    -2(\eta+\delta A^\dagger) & -(\eta+\rho)
  \end{pmatrix}.
\]
Hence the only nontrivial stability condition is
\[
  (\beta-\mu)+(\delta-\beta)A^\dagger<0.
\]
Substituting \(A^\dagger=\eta/(\eta+\rho)\) yields
\eqref{eq:RLA_symm}. Equivalently,
\[
  \mathcal R_{\mathrm{LA}}^{\mathrm{sym}}
  =
  \frac{C^\dagger\beta+A^\dagger\delta}{\mu},
\]
so the local growth criterion is a weighted blend of centrist recruitment and
mobilisation from alienation.

Now let \((P,A)\) be a positive equilibrium. The second equation of
\eqref{eq:model_symm} gives
\[
  \eta(1-2P-A)=(2\delta P+\rho)A,
\]
which is equivalent to \eqref{eq:AofP_symm}. Since \(P>0\), the first equation
requires
\[
  (\beta-\mu)-2\beta P+(\delta-\beta)A=0.
\]
Substituting \eqref{eq:AofP_symm} and clearing the positive denominator
\(\eta+\rho+2\delta P\) yields the quadratic equation
\eqref{eq:Q_symm}--\eqref{eq:Q_explicit_symm}. The endpoint identities in
\eqref{eq:Q_endpoints_symm} follow by direct substitution.

Finally, write
\[
  \dot P=P\,G(P,A),
  \qquad
  G(P,A):=(\beta-\mu)-2\beta P+(\delta-\beta)A.
\]
At a positive equilibrium one has \(P>0\) and \(G(P,A)=0\). Therefore
\[
  \frac{\partial \dot P}{\partial P}
  =
  G(P,A)+P\frac{\partial G}{\partial P}
  =
  -2\beta P,
  \qquad
  \frac{\partial \dot P}{\partial A}
  =
  P\frac{\partial G}{\partial A}
  =
  P(\delta-\beta).
\]
The derivatives of \(\dot A\) are
\[
  \frac{\partial \dot A}{\partial P}=-2(\eta+\delta A),
  \qquad
  \frac{\partial \dot A}{\partial A}=-(\eta+\rho+2\delta P),
\]
which proves \eqref{eq:J_positive_symm}.
\end{proof}

\subsection{Proof of Proposition~\ref{prop:all_positive_symm}}
\label{app:prop_all_positive_symm}

\begin{proof}
Under \eqref{eq:symm_assumptions}, the frozen equations for \(L\) and \(R\) are
\[
  \dot L=(\alpha C+\delta A-\mu)L+\gamma CR,
  \qquad
  \dot R=\gamma CL+(\alpha C+\delta A-\mu)R.
\]
At a positive equilibrium \(E^*\), adding the two equations gives
\[
  0=\dot L+\dot R
  =
  \bigl((\alpha+\gamma)C^*+\delta A^*-\mu\bigr)(L^*+R^*)
  =
  (\beta C^*+\delta A^*-\mu)(L^*+R^*).
\]
Since \(L^*+R^*>0\), it follows that
\begin{equation}\label{eq:sum_balance_symm}
  \beta C^*+\delta A^*-\mu=0.
\end{equation}
Subtracting the two equations yields
\[
  0=\dot L-\dot R
  =
  \bigl((\alpha-\gamma)C^*+\delta A^*-\mu\bigr)(L^*-R^*).
\]
Using \eqref{eq:sum_balance_symm} and \(\beta=\alpha+\gamma\), we obtain
\[
  (\alpha-\gamma)C^*+\delta A^*-\mu
  =
  (\beta-2\gamma)C^*+\delta A^*-\mu
  =
  -2\gamma C^*<0,
\]
because \(\gamma>0\) and \(C^*>0\). Therefore \(L^*-R^*=0\), so \(L^*=R^*\).
\end{proof}

\subsection{Proof of Proposition~\ref{prop:lift_symm_stability}}
\label{app:prop_lift_symm_stability}

\begin{proof}
For the antisymmetric variable \(H=L-R\), the full frozen symmetric system
satisfies
\[
  \dot H=\bigl((\alpha-\gamma)C+\delta A-\mu\bigr)H.
\]
At a positive symmetric equilibrium \(E^*=(P^*,P^*,A^*)\), the reduced
\(P\)-equation implies
\[
  (\beta-\mu)-2\beta P^*+(\delta-\beta)A^*=0.
\]
Since \(C^*=1-2P^*-A^*\), this is equivalent to
\[
  \beta C^*+\delta A^*-\mu=0.
\]
Using \(\beta=\alpha+\gamma\), we obtain
\[
  \lambda_H
  =
  (\alpha-\gamma)C^*+\delta A^*-\mu
  =
  (\beta-2\gamma)C^*+\delta A^*-\mu
  =
  -2\gamma C^*<0,
\]
because \(\gamma>0\) and \(C^*>0\) at every positive symmetric equilibrium.
Thus the antisymmetric mode is stable for any positive equilibrium of the full
frozen symmetric system. The remaining two eigenvalues coincide with those of
the planar Jacobian \(J(P^*,A^*)\), which proves the claim.
\end{proof}

\subsection{Proof of Theorem~\ref{thm:delay}}
\label{app:thm_delay}

\begin{proof}
Set
\[
  e(t):=\theta(t)-\theta_\star(t).
\]
Subtracting \(\dot\theta_\star(t)\) from \eqref{eq:theta_relax} yields
\[
  \dot e(t)=-\kappa_\theta e(t)-\dot\theta_\star(t)
  \qquad\text{for a.e. }t\in I.
\]
Variation of constants and \eqref{eq:target_speed_delay} give
\begin{equation}\label{eq:error_bound_app}
  \|e(t)\|
  \le
  e^{-\kappa_\theta(t-t_0)}\|e(t_0)\|
  +
  \frac{V_\star}{\kappa_\theta}\bigl(1-e^{-\kappa_\theta(t-t_0)}\bigr)
  \le
  E_\star
  \qquad\text{for all }t\in I.
\end{equation}
By Lemma~\ref{lem:RLA_regular},
\[
  |r_{\mathrm{eff}}(t)-r_\star(t)|\le L_{\mathcal R}E_\star,
\]
hence
\begin{equation}\label{eq:signed_gap_app}
  \sigma\bigl(r_{\mathrm{eff}}(t)-1\bigr)
  \ge
  \sigma\bigl(r_\star(t)-1\bigr)-L_{\mathcal R}E_\star
  \qquad\text{for all }t\in I.
\end{equation}
For \(t\in[t_0,t_\star)\), \eqref{eq:signed_crossing_delay} and
\eqref{eq:scalar_lag_delay} give
\(\sigma(r_{\mathrm{eff}}(t)-1)<0\), so
\begin{equation}\label{eq:no_early_app}
  t_{\mathrm{eff}}^\sigma\ge t_\star.
\end{equation}
Set \(\tau_\star:=L_{\mathcal R}E_\star/m_\star\). By
\eqref{eq:delay_smallness}, \(t_\star+\tau_\star\le t_\star+\Delta_\star\).
Applying \eqref{eq:transversal_delay} and \eqref{eq:signed_gap_app} at
\(t=t_\star+\tau_\star\),
\[
  \sigma\bigl(r_{\mathrm{eff}}(t_\star+\tau_\star)-1\bigr)
  \ge
  m_\star\tau_\star - L_{\mathcal R}E_\star = 0.
\]
Continuity of \(r_{\mathrm{eff}}\) then gives
\(t_{\mathrm{eff}}^\sigma\le t_\star+\tau_\star\).
Combining with \eqref{eq:no_early_app} yields \eqref{eq:delay_bound}.
\end{proof}


\end{document}